\documentclass[final]{siamltex}
\usepackage{amsfonts,amsmath,amssymb}
\usepackage{mathrsfs,mathtools}
\usepackage{enumerate}
\usepackage{xspace,mydef2}
\usepackage{hyperref}
\usepackage{esint}
\usepackage{graphicx,epstopdf}
\usepackage{textcomp}
\usepackage{calc}
\usepackage{nicefrac}
\usepackage{stmaryrd}
\usepackage{bbm}
\usepackage[caption=false]{subfig}
\DeclareMathAlphabet{\mathpzc}{OT1}{pzc}{m}{it}

\DeclareFontEncoding{FMS}{}{}
\DeclareFontSubstitution{FMS}{futm}{m}{n}
\DeclareFontEncoding{FMX}{}{}
\DeclareFontSubstitution{FMX}{futm}{m}{n}
\DeclareSymbolFont{fouriersymbols}{FMS}{futm}{m}{n}
\DeclareSymbolFont{fourierlargesymbols}{FMX}{futm}{m}{n}
\DeclareMathDelimiter{\VERT}{\mathord}{fouriersymbols}{152}{fourierlargesymbols}{147}

\usepackage[notcite,notref]{showkeys} % To see crossreferences.
\usepackage{color}
\newcommand{\EO}[1]{\textcolor{black}{#1}}

\newtheorem{remark}[theorem]{Remark}
\numberwithin{equation}{section}

\newcommand{\Nin}{\,{\mbox{\,\raisebox{6.0pt} {\tiny$\circ$} \kern-11.1pt}\N }}
\newcommand{\Ninn}{{\mbox{\,\raisebox{4.5pt} {\tiny$\circ$} \kern-8.8pt}\N }}
\newcommand{\osc}{{\textup{\textsf{osc}}}}

\title{An adaptive finite element method for the sparse optimal control of fractional diffusion\thanks{EO is partially supported by CONICYT through FONDECYT project 11180193.}}

\author{Enrique Ot\'arola\thanks{Departamento de Matem\'atica, 
Universidad T\'ecnica Federico Santa Mar\'ia, Valpara\'iso, Chile. 
\texttt{enrique.otarola@usm.cl}.}}

\date{Draft version of \today.}

\begin{document}

\maketitle
\begin{abstract}
We propose and analyze an a posteriori error estimator for a PDE--constrained optimization problem involving a nondifferentiable cost functional, fractional diffusion, and control--constraints. We realize fractional diffusion as the Dirichlet-to-Neumann map for a nonuniformly PDE and propose an equivalent optimal control problem with a local state equation. For such an equivalent problem, we design an a posteriori error estimator which can be defined as the sum of four contributions: two contributions related to the approximation of the state and adjoint equations and two contributions that account for the discretization of the control variable and its associated subgradient. The contributions related to the discretization of the state and adjoint equations rely on anisotropic error estimators in weighted Sobolev spaces. We prove that the proposed a posteriori error estimator is locally efficient and, under suitable assumptions, reliable. We design an adaptive scheme that yields, for the examples that we perform, optimal experimental rates of convergence.
\end{abstract}



%%%%%%%%%%%%%%%%%%%%%%%%%%%%%%%%%%%%%%%%%%%%%%%%%%%%%%%%%%%%%%%%%%%%%%%%%%%%%%%%%%%%%%
\section{Introduction}
\label{sec:introduccion}
%%%%%%%%%%%%%%%%%%%%%%%%%%%%%%%%%%%%%%%%%%%%%%%%%%%%%%%%%%%%%%%%%%%%%%%%%%%%%%%%%%%%%%
The main goal of this work is the design and study of a posteriori error estimates for an optimal control problem that entails the minimization of a nondifferentiable cost functional, a state equation that involves the spectral fractional powers of the Dirichlet Laplace operator, and constraints on the control variable. To be precise, let $\Omega \subset \R^n$ ($n \geq 1$) be an open and bounded polytopal domain with Lipschitz boundary $\partial \Omega$, $s \in (0,1)$, and $\usf_d: \Omega \rightarrow \mathbb{R}$ be a desired state. For parameters $\sigma>0$ and $\nu >0$, we define the nonsmooth cost functional
\begin{equation}
\label{def:J}
J(\usf,\zsf) := \frac{1}{2}\| \usf - \usf_{d} \|^2_{L^2(\Omega)} + 
\frac{\sigma}{2} \| \zsf\|^2_{L^2(\Omega)} + \nu \| \zsf \|_{L^1(\Omega)}.
\end{equation}
We will be interested in the numerical approximation of the following nondifferentiable PDE--constrained optimization problem: Find
\begin{equation}
\label{def:minJ}
 \text{min }J(\usf,\zsf),
\end{equation}
subject to the \emph{nonlocal state equation}
\begin{equation}
\label{eq:fractional}
\Laps \usf = \zsf  \text{ in } \Omega, 
%\qquad \usf = 0   \text{ on } \partial \Omega, \\
\end{equation}
and the \emph{control constraints}
\begin{equation}
 \label{eq:control_constraints}
\asf \leq \zsf(x') \leq \bsf \quad\textrm{a.e.~~} x' \in \Omega . 
\end{equation}
For $s \in (0,1)$, the operator $\Laps$ denotes the \emph{spectral} fractional powers of the Dirichlet Laplace operator; the so--called \emph{spectral fractional Laplacian}. We must immediately comment that this definition and the classical one that is based on a pointwise integral formula \cite{Landkof,Silvestre:2007,Stein} do not coincide; their difference is positive and positivity preserving \cite{MR3246044}. In addition, we also comment that, since we are interested in the nonsmooth scenario, we will assume, in \eqref{eq:control_constraints}, that the control bounds $\asf, \bsf \in \mathbb{R}$ satisfy that $\asf < 0 < \bsf$. We refer the reader to \cite[Remark 2.1]{CHW:12} for a discussion.
 
The efficient approximation of problems involving the spectral fractional Laplacian carries two essential difficulties. The first, and most important, is that $(-\Delta)^s$ is a nonlocal operator \cite{CT:10,CS:07,CDDS:11}. The second feature is the lack of boundary regularity \cite{MR3489634}, which leads to reduced convergence rates \cite{BMNOSS:17,NOS}. In fact, as \cite[Theorem 1.3]{MR3489634} shows, if $\partial \Omega$ is sufficiently smooth, then the solution $\usf$ to problem \eqref{eq:fractional} behaves like
\begin{equation}
\begin{aligned}
\label{eq:u_singular_at_the_boundary}
 \usf(x') &\approx \textrm{dist}(x',\partial \Omega)^{2s} + \vsf(x'), \quad &s \in (0,\tfrac{1}{2}),
 \\
 \usf(x') & \approx \textrm{dist}(x',\partial \Omega)+ \vsf(x'), \quad  &s \in (\tfrac{1}{2},1),
 \end{aligned}
\end{equation}
where $\textrm{dist}(x',\partial \Omega)$ denotes the distance from $x'$ to $\partial \Omega$ and $\vsf$ is a smooth function. The case $s = \tfrac{1}{2}$ is exceptional:
\begin{equation}
\label{eq:u_singular_at_the_boundary_1/2}
\usf(x') \approx \textrm{dist}(x',\partial \Omega)|\log(\textrm{dist}(x',\partial \Omega))|+ \vsf(x'), 
\end{equation}
where $\vsf$ is, again, a smooth function; see \cite{MR1204855} for $\Omega \subset \mathbb{R}^2$ and $\partial \Omega$ smooth. 

The aforementioned nonlocality difficulty can be overcame with the localization results by Caffarelli and Silvestre \cite{CS:07}. When $\Omega = \mathbb{R}^n$, the authors of \cite{CS:07} proved that any power of the fractional Laplacian can be realized as the Dirichlet-to-Neumann map for an extension problem posed in the upper half--space $\R^{n+1}_+$. A similar extension property is valid for the spectral fractional Laplacian in a bounded domain $\Omega$ \cite{CT:10,CDDS:11}. The latter extension involves a local but nonuniformly elliptic PDE formulated in the semi--infinite cylinder $\C = \Omega \times (0,\infty)$:
\begin{equation}
\label{eq:alpha_harm}
  \DIV\left( y^\alpha \nabla \ue \right) = 0 \text{ in } \C, \quad \ue = 0 \text{ on } \partial_L \C, \quad
  \partial_{\nu^\alpha} \ue = d_s \zsf \text{ on } \Omega \times \{0\},
\end{equation}
where $\partial_L \C= \partial \Omega \times [0,\infty)$ corresponds to the lateral boundary of $\C$ and $d_s = 2^{\alpha}\Gamma(1-s)/\Gamma(s)$. The parameter $\alpha$ is defined as $\alpha = 1-2s \in (-1,1)$ and the conormal exterior derivative of $\ue$ at $\Omega \times \{ 0 \}$ is
\begin{equation}
\label{def:lf}
\partial_{\nu^\alpha} \ue = -\lim_{y \rightarrow 0^+} y^\alpha \ue_y;
\end{equation}  
the limit is understood in the distributional sense.
We shall refer to $y$ as the \emph{extended variable} and to the dimension $n+1$, in $\R_+^{n+1}$, the \emph{extended dimension} of problem \eqref{eq:alpha_harm}. With the extension $\ue$ at hand, we thus introduce the fundamental result by Caffarelli and Silvestre \cite{CT:10,CS:07,CDDS:11}: the Dirichlet-to-Neumann map of problem \eqref{eq:alpha_harm} and the spectral fractional Laplacian are related by 
$
  d_s \Laps \usf = \partial_{\nu^\alpha} \ue \textrm{ in } \Omega.
$

The use of the extension problem \eqref{eq:alpha_harm} for the discretization of the spectral fractional Laplacian was first used in \cite{NOS}. 
%In such a work the authors introduce the following solution scheme: Given a datum $\zsf$, solve, on the basis of a finite element discretization, the extension problem \eqref{eq:alpha_harm} and obtain $V$, thus set $U = \tr V$ and arrive at an approximation of the solution $\usf$ of problem \eqref{eq:fractional}. 
The main advantage of the scheme proposed in \cite{NOS} is that it involves the resolution of the local problem \eqref{eq:alpha_harm} and thus its implementation uses basic ingredients of finite element analysis; its analysis, however, involves asymptotic estimates of Bessel functions \cite{Abra}, to derive regularity estimates in weighted Sobolev spaces, elements of harmonic analysis \cite{Javier,Muckenhoupt}, and an anisotropic polynomial interpolation theory in weighted Sobolev spaces \cite{DL:05, NOS2}. Such an interpolation theory allows for tensor product elements that exhibit an anisotropic feature in the extended dimension, 
%a large aspect ratio in $y$ (anisotropy), 
that is in turn needed to compensate the singular behavior of the solution $\ue$ in the extended variable $y$ \cite[Theorem 2.7]{NOS}, \cite[Theorem 4.7]{BMNOSS:17}.

PDE--constrained optimization problems for fractional diffusion have been considered in a number of works \cite{MR3429730,MR3850351,DGO,GO,MR3702421}. Recently, the authors of \cite{OS:Sparse} have provided an a priori error analysis for the sparse optimal control problem \eqref{def:minJ}--\eqref{eq:control_constraints}. This was mainly motivated by the following considerations:
\begin{enumerate}[$\bullet$]
 \item Practitioners claim that the fractional Laplacian seems to better describe many processes. A rather incomplete list of problems where fractional diffusion appears includes finance \cite{MR2064019,PH:97}, turbulent flow \cite{chen23}, quasi--geostrophic flows models \cite{CV:10}, models of anomalous thermoviscous behaviors \cite{doi:10.1121/1.1646399}, biophysics \cite{bio}, nonlocal electrostatics \cite{ICH}, image processing \cite{GH:14}, peridynamics \cite{MR3023366,MR3618711}, and many others.  It is then only natural that interest in efficient approximation schemes for these problems arises and that one might be interested in their control.
 \item The cost functional $J$ involves an $L^1(\Omega)$--control cost term that leads to sparsely supported optimal controls \cite{CHW:12,MR2556849,MR2826983}; a desirable feature, for instance, in the optimal placement of discrete actuators \cite{MR2556849}. 
\end{enumerate}

In \cite{OS:Sparse}, the authors first consider an equivalent optimal control problem that involves the local elliptic PDE \eqref{eq:alpha_harm} as state equation. Second, since \eqref{eq:alpha_harm} is posed on the semi--infinite cylinder $\C = \Omega \times (0,\infty)$, they propose a truncated optimal control problem on $\C_{\Y} =  \Omega \times (0,\Y)$ and derive an exponential error estimate with respect to the truncation parameter $\Y$. Then, they propose a scheme to approximate the truncated optimal control problem: the first--degree finite element approximation of \cite{NOS} for the state and adjoint equations and piecewise constant approximation for the optimal control variable. The derived a priori error estimate reads as follows: Given $\usf_d \in \Ws$ and $\asf, \bsf \in \mathbb{R}$ such that $\asf < 0 < \bsf$, if $\Omega$ is convex, then
\begin{equation}
\label{eq:a_priori}
 \| \bar \zsf -\bar Z \|_{L^2(\Omega)} \lesssim |\log N|^{2s}N^{-\frac{1}{n+1}},
\end{equation}
where $\bar Z$ corresponds to the optimal control variable of the scheme of \cite[Section 6]{OS:Sparse} and $N$ denotes the total number of degrees of freedom of the underlying mesh. The adaptive finite element method (AFEM) that we propose in our work is thus motivated, in addition to the search of a numerical scheme that efficiently solves \eqref{def:minJ}--\eqref{eq:control_constraints} with relatively modest computational resources, by the following considerations:
\begin{enumerate}[$\bullet$]
 \item The a priori error estimate \eqref{eq:a_priori} requires the convexity of the domain $\Omega$ and compatibility conditions on the desired state $\usf_d$ that are expressed as $\usf_d \in \Ws$: \emph{it is required that $\usf_d$ vanishes on $\partial \Omega$ for $s \in (0,\frac{1}{2}]$}. If these conditions do not hold then, the a priori error estimate \eqref{eq:a_priori} is not longer valid. In particular, the violation of the latter condition implies that the adjoint state $\bar \psf$ will behave as \eqref{eq:u_singular_at_the_boundary}--\eqref{eq:u_singular_at_the_boundary_1/2}. On the other hand, if the first condition (the convexity of the domain) is violated, then both the state and adjoint state variables may exhibit singularities in the direction of the $x'$ variables. An efficient technique to solve \eqref{def:minJ}--\eqref{eq:control_constraints} must thus resolve both of the aforementioned approximation issues.
 
 \item The sparsity term $\| \usf \|_{L^1(\Omega)}$, in the cost functional \eqref{def:minJ}, yield an optimal control that is non--zero only in sets of small support in $\Omega$  \cite{CHW:12,MR2556849,MR2826983}. It is then natural, to efficiently resolve such a behavior, the consideration of AFEMs.
\end{enumerate}

We organize our exposition as follows: In Section \ref{sec:Prelim} we recall the definition of the spectral fractional Laplacian, present the fundamental result by Caffarelli and Silvestre \cite{CS:07}, and recall elements from convex analysis. In Section \ref{sec:apriori_control} we recall the numerical scheme proposed in \cite{OS:Sparse} and review its a priori error analysis. Section \ref{subsec:ideal_a_posteriori}, that is a highlight of our contribution, is dedicated to the design and analysis of an ideal a posteriori error estimator for \eqref{def:minJ}--\eqref{eq:control_constraints} that is equivalent to the error. Since, the aforementioned estimator is not computable, we propose, in Section \ref{subsec:computable_a_posteriori} a computable one and show that is equivalent, under suitable assumptions, to the error up to oscillation terms. Finally, in Section \ref{sec:numerics} we design an AFEM, comment on some implementation details pertinent to the problem, and present numerical experiments that yield optimal experimental rates of convergence.
%%%%%%%%%%%%%%%%%%%%%%%%%%%%%%%%%%%%%%%%%%%%%%%%%%%%%%%%%%%%%%%%%%%%%%%%%%%%%%%%%%%%%%%%%%%%%%%%%%%%%%%%%%%%%%%%%%%%%%%%%%%%%%%%%%%%%%%%%%%%%%%%%%%%%%%%%%%%%%%%%%%%%%%%%%%%%
\section{Notation and preliminaries}
\label{sec:Prelim}
%%%%%%%%%%%%%%%%%%%%%%%%%%%%%%%%%%%%%%%%%%%%%%%%%%%%%%%%%%%%%%%%%%%%%%%%%%%%%%%%%%%%%%%%%%%%%%%%%%%%%%%%%%%%%%%%%%%%%%%%%%%%%%%%%%%%%%%%%%%%%%%%%%%%%%%%%%%%%%%%%%%%%%%%%%%%%
We adopt the notation of \cite{NOS,MR3259962}. Besides the semi--infinite cylinder $\C= \Omega \times (0,\infty)$, we introduce, for $\Y > 0$, the truncated cylinder $\C_{\Y}= \Omega \times (0,\Y)$
and its lateral boundary $\partial_L\C_\Y = \partial \Omega \times (0,\Y)$. 

For $ x\in \R_+^{n+1}$, we write 
$
x = (x',y)
$
with  $x' \in \R^n$ and $y \in (0,\infty)$. 

The parameter $\alpha \in (-1,1)$ and the power $s$ of the spectral fractional Laplacian $\Laps$ are related by the formula $\alpha = 1 -2s$.

The relation $a \lesssim b$ indicates that $a \leq Cb$, with a constant $C$ which is independent of $a$ and $b$ and the size of the elements in the mesh. The value of the constant $C$ might change at each occurrence.

%%%%%%%%%%%%%%%%%%%%%%%%%%%%%%%%%%%%%%%%%%%%%%%%%%%%%%%%%%%%%%%%%%%%%%%%%%%%%%%%%%%%%%%%%%%%%%%%%%%%%%%%%%%%%%%%%%%%%%%%%%%%%%%%%%%%%%%%%%%%%%%%%%%%%%%%%%%%%%%%%%%%%%%%%%%%%
\subsection{The fractional Laplace operator}
\label{subsec:fractional_Laplacian}
%%%%%%%%%%%%%%%%%%%%%%%%%%%%%%%%%%%%%%%%%%%%%%%%%%%%%%%%%%%%%%%%%%%%%%%%%%%%%%%%%%%%%%%%%%%%%%%%%%%%%%%%%%%%%%%%%%%%%%%%%%%%%%%%%%%%%%%%%%%%%%%%%%%%%%%%%%%%%%%%%%%%%%%%%%%%%
%To define $\Laps$ we invoke spectral theory \cite{BS}. 
Since $-\Delta: \mathcal{D}(-\Delta) \subset L^2(\Omega) \rightarrow L^2(\Omega)$ is an unbounded, positive, and closed operator with dense domain $\mathcal{D}(-\Delta) = H^2(\Omega) \cap H_0^1(\Omega)$ and its inverse is compact, the eigenvalue problem: 
%Find $(\lambda,\varphi) \in \mathbb{R} \times H_0^1(\Omega) \setminus \{ 0 \}$ such that
\[
(\lambda,\varphi) \in \mathbb{R} \times H_0^1(\Omega) \setminus \{ 0 \}:
\quad
( \nabla \varphi, \nabla v)_{L^2(\Omega)} = \lambda (\varphi,v)_{L^2(\Omega)} \quad \forall v \in H_0^1(\Omega)
\]
has a countable collection of eigenpairs $\{ \lambda_k, \varphi_k \}_{k \in \mathbb{N}} \subset \mathbb{R}^+ \times H_0^1(\Omega)$, with real eigenvalues enumerated in increasing order, counting multiplicities. In addition, $\{ \varphi_k \}_{k \in \mathbb{N}}$ is an orthogonal basis of $H_0^1(\Omega)$ and an orthonormal basis of $L^2(\Omega)$ \cite{BS}. With these eigenpairs at hand, we define, for $s \geq 0$, the fractional Sobolev space
\begin{equation*}
\label{def:Hs}
  \Hs = \left\{ w = \sum_{k=1}^\infty w_k \varphi_k: \| w \|_{\Hs}^2 =
  \sum_{k=1}^{\infty} \lambda_k^s w_k^2 < \infty \right\},
\end{equation*}
\EO{where, for  $k \in \mathbb{N}$, $w_k = ( w , \varphi_k)_{L^2(\Omega)}$.} We denote by $\Hsd$ the dual space of $\Hs$. The duality pairing between $\Hs$ and $\Hsd$ will be denoted by $\langle \cdot, \cdot \rangle$. 

We define, for $s \in (0,1)$, the \emph{spectral} fractional Laplacian $\Laps$ as 
\begin{equation*}
\Laps: \Hs \rightarrow \Hsd,
\quad
  (-\Delta)^s w  := \sum_{k=1}^\infty \lambda_k^{s} w_k \varphi_k, \quad w_k = ( w , \varphi_k)_{L^2(\Omega)}, \quad k \in \mathbb{N}.
\end{equation*} 

\subsection{An extension property}
\label{subsec:extension_property}

The operator in \eqref{eq:alpha_harm} is in divergence form and thus amenable to variational techniques. However, since the weight $y^{\alpha}$ either blows up, for $-1<\alpha<0$, or degenerates, for $0<\alpha<1$, as $y \downarrow 0$, such a local operator is nonuniformly elliptic; the case $\alpha = 0$ is exceptional and corresponds to the regular harmonic extension \cite{CT:10}. This entails dealing with Lebesgue and Sobolev spaces with the weight $y^{\alpha}$ for $\alpha \in (-1,1)$ \cite{CS:07,CDDS:11}. 

Let $D \subset \R^{n+1}$ be open. We define $L^2(|y|^{\alpha},D)$ as the Lebesgue space for the measure $|y|^{\alpha} \diff x$. We also define the weighted Sobolev space
$
 H^1(|y|^{\alpha},D) := \{ w \in L^2(|y|^{\alpha},D): | \nabla w | \in L^2(|y|^{\alpha},D) \},
$
and the norm
\begin{equation}
\label{wH1norm}
\| w \|_{H^1(|y|^{\alpha},D)} =
\left(  \| w \|^2_{L^2(|y|^{\alpha},D)} + \| \nabla w \|^2_{L^2(|y|^{\alpha},D)} \right)^{\frac{1}{2}}.
\end{equation}
Since $\alpha = 1-2s \in (-1,1)$, $|y|^\alpha$ belongs to Muckenhoupt class $A_2(\R^{n+1})$ \cite{Javier,FKS:82,GU,Muckenhoupt,Turesson}. This, in particular, implies that $H^1(|y|^{\alpha},D)$ is Hilbert and that $C^{\infty}(D) \cap H^1(|y|^{\alpha},D)$ is dense in $H^1(|y|^{\alpha},D)$ (cf.~\cite[Proposition 2.1.2, Corollary 2.1.6]{Turesson} and \cite[Theorem~1]{GU}).

To seek for a weak solution to problem \eqref{eq:alpha_harm}, we introduce the weighted space
\begin{equation*}
  \label{HL10}
  \HL(y^{\alpha},\C) := \left\{ w \in H^1(y^\alpha,\C): w = 0 \textrm{ on } \partial_L \C\right\},
\end{equation*}
and notice that the following \emph{weighted Poincar\'e inequality} holds \cite[ineq. (2.21)]{NOS}:
\begin{equation*}
\label{Poincare_ineq}
\| w \|_{L^2(y^{\alpha},\C)} \lesssim \| \nabla w \|_{L^2(y^{\alpha},\C)}
\quad \forall w \in \HL(y^{\alpha},\C).
\end{equation*}
This implies that the seminorm on $\HL(y^{\alpha},\C)$ is equivalent to \eqref{wH1norm}. For $w \in H^1(y^{\alpha},\C)$, $\tr w$ denotes its trace onto $\Omega \times \{ 0 \}$. We recall (\cite[Prop.~2.5]{NOS} and \cite[Prop.~2.1]{CDDS:11})
\begin{equation}
\label{Trace_estimate}
\tr \HL(y^{\alpha},\C) = \Hs,
\qquad
  \|\tr w\|_{\Hs} \leq C_{\tr} \| w \|_{\HLn(y^{\alpha},\C)},
  \quad
 C_{\tr} > 0.
\end{equation}
We mention that $C_{\tr} \leq d_s^{-\frac{1}{2}}$ with $d_s = 2^{\alpha}\Gamma(1-s)/\Gamma(s)$ \cite[Section 2.3]{CNOS2}. This property will be of importance in the a posteriori error analysis that we will perform.

Define the bilinear form
\begin{equation}
 \label{eq:a}
 a: \HL(y^{\alpha},\C) \times \HL(y^{\alpha},\C) \rightarrow \mathbb{R}, \quad a(w,\phi):= \frac{1}{d_s} \int_{\C} y^{\alpha} \nabla w \cdot \nabla \phi \diff x.
\end{equation}
The weak formulation of problem \eqref{eq:alpha_harm} thus reads: Find $\ue \in \HL(y^{\alpha},\C)$ such that
\begin{equation}
 \label{eq:weak_alpha_harm}
 a(\ue,\phi) = \langle \zsf, \phi \rangle \quad \forall \phi \in \HL(y^{\alpha},\C).
\end{equation}

We conclude this section with the fundamental result by Caffarelli and Silvestre \cite{CT:10,CS:07,CDDS:11}: Let $\usf$ solve \eqref{eq:fractional}. If $\ue \in \HL(y^{\alpha},\C)$ solves \eqref{eq:weak_alpha_harm}, then $\usf = \tr \ue$ and
\[
 d_s \Laps \usf = \partial_{\nu^{\alpha}} \ue \textrm{ in } \Omega.
\]

\subsection{Convex functions and subdifferentials}
\label{subsec:subdifferentials}
Let $S$ be a real and normed vector space. Let $\chi: S \rightarrow \R \cup \{\infty\}$ be a convex and proper function and let $v \in S$ be such that $\chi(v) < \infty$. A \emph{subgradient} of $\chi$ at $v$ is an element $v^\star \in S^{\star}$ that satisfies
\begin{equation}
 \label{eq:subgradient}
\langle v^\star,w - v \rangle_{S^\star,S} \leq \chi(w) - \chi(v) \quad \forall w \in S,
\end{equation}
where $\langle \cdot, \cdot \rangle_{S^\star,S}$ denotes the duality pairing between $S^{\star}$ and $S$. We denote by $\partial \chi (v)$ the set of all subgradients of $\chi$ at $v$: the so--called  \emph{subdifferential} of $\chi$ at $v$. \EO{Since $\chi$ is convex,} $\partial \chi(v) \neq \emptyset$ for all points $v$ in the interior of the effective domain of $\chi$. We conclude with the following property: the subdifferential is monotone, \ie
\begin{equation}
\label{eq:subdiff_monotone}
  \langle v^\star - w^\star, v - w \rangle_{S^\star,S} \geq 0 \quad \forall v^\star \in \partial \chi(v),\ \forall w^\star \in \partial\chi(w).
\end{equation}
The reader is referred to \cite{MR1058436,MR2330778} for a detailed treatment on convex analysis.

%%%%%%%%%%%%%%%%%%%%%%%%%%%%%%%%%%%%%%%%%%%%%%%%%%%%%%%%%%%%%%%%%%%%%%%%%%%%%%%%%%%%%%%%%%%%%%%%%%%%%%%%%%%%%%%%%%%%%%%%%%%%%%%%%%%%%%%%%%%%%%%%%%%%%%%%%%%%%%%%%%%%%%%%%%
\section{A priori error estimates}
\label{sec:apriori_control}
%%%%%%%%%%%%%%%%%%%%%%%%%%%%%%%%%%%%%%%%%%%%%%%%%%%%%%%%%%%%%%%%%%%%%%%%%%%%%%%%%%%%%%%%%%%%%%%%%%%%%%%%%%%%%%%%%%%%%%%%%%%%%%%%%%%%%%%%%%%%%%%%%%%%%%%%%%%%%%%%%%%%%%%%%%

In this section we briefly review the a priori error analysis for the fully discrete approximation of the optimal control problem \eqref{def:minJ}--\eqref{eq:control_constraints} proposed and 
investigated in \cite{OS:Sparse}. We also make clear the limitations of such a priori error analysis, thereby justifying the quest for an a posteriori error analysis.

%%%%%%%%%%%%%%%%%%%%%%%%%%%%%%%%%%%%%%%%%%%%%%%%%%%%%%%%%%%%%%%%%%%%%%%%%%%%%%%%%%%%%%%%%%%%%%%%%%%%%%%%%%%%%%%%%%%%%%%%%%%%%%%%%%%%%%%%%%%%%%%%%%%%%%%%%%%%%%%%%%%%%%%%%%
\subsection{The extended optimal control problem}
\label{subsec:extended}
%%%%%%%%%%%%%%%%%%%%%%%%%%%%%%%%%%%%%%%%%%%%%%%%%%%%%%%%%%%%%%%%%%%%%%%%%%%%%%%%%%%%%%%%%%%%%%%%%%%%%%%%%%%%%%%%%%%%%%%%%%%%%%%%%%%%%%%%%%%%%%%%%%%%%%%%%%%%%%%%%%%%%%%%%%

We will consider a solution technique for \eqref{def:minJ}--\eqref{eq:control_constraints} that relies on the discretization of an equivalent problem: the \emph{extended optimal control problem}, which has as a main advantage its local nature. To present it, we first define the set of \emph{admissible controls}:
\begin{equation}
 \label{def:Zad}
 \Zad= \{ \wsf \in L^2(\Omega): \asf \leq \wsf(x') \leq \bsf \textrm{~~a.e~~}  x' \in \Omega \},
\end{equation}
where $\asf$ and $\bsf$ are real and satisfy the property $\asf < 0 <  \bsf$; see\cite[Remark 2.1]{CHW:12}. The extended optimal control problem reads as follows: Find
\begin{equation}
 \text{min} \{ J(\tr \ue,\zsf): \ue \in \HL(y^{\alpha},\C), \zsf \in \Zad \}
\end{equation}
subject to the \emph{linear} and nonuniformly elliptic state equation
\begin{equation}
\label{eq:alpha_harm_weak}
a(\ue,\phi) = \langle \zsf, \tr \phi \rangle \quad \forall \phi \in \HL(y^{\alpha},\C).
\end{equation}
We recall that $J$ is defined as in \eqref{def:minJ}, with $\sigma ,\nu >0$ and $\usf_d \in L^2(\Omega)$. This problem admits a unique optimal pair $(\oue,\ozsf) \in \HL(y^{\alpha},\C) \times \Zad$. More importantly, it is equivalent to the optimal control problem \eqref{def:minJ}--\eqref{eq:control_constraints}: $\tr \oue = \ousf$.

%%%%%%%%%%%%%%%%%%%%%%%%%%%%%%%%%%%%%%%%%%%%%%%%%%%%%%%%%%%%%%%%%%%%%%%%%%%%%%%%%%%%%%%%%%%%%%%%%%%%%%%%%%%%%%%%%%%%%%%%%%%%%%%%%%%%%%%%%%%%%%%%%%%%%%%%%%%%%%%%%%%%%%%%%%
\subsection{The truncated optimal control problem}
\label{subsec:truncated}
%%%%%%%%%%%%%%%%%%%%%%%%%%%%%%%%%%%%%%%%%%%%%%%%%%%%%%%%%%%%%%%%%%%%%%%%%%%%%%%%%%%%%%%%%%%%%%%%%%%%%%%%%%%%%%%%%%%%%%%%%%%%%%%%%%%%%%%%%%%%%%%%%%%%%%%%%%%%%%%%%%%%%%%%%%

The extended optimal control problem involves the state equation \eqref{eq:alpha_harm_weak}, which is posed on $\C = \Omega \times (0,\infty)$. Consequently, it cannot be directly approximated with standard finite element techniques. However, in view of the exponential decayment of the optimal state in the extended variable \cite[Proposition 3.1]{NOS}, it is suitable to propose the following \emph{truncated optimal control problem}. Find
\[
 \text{min} \{ J(\tr v,\rsf): v \in \HL(y^{\alpha},\C_\Y), \rsf \in \Zad \}
\]
subject to the \emph{truncated} state equation
\begin{equation}
\label{eq:alpha_harm_truncated}
  a_\Y(v,\phi) = \langle \rsf, \tr \phi \rangle
    \quad \forall \phi \in \HL(y^{\alpha},\C_\Y).
\end{equation}
Here,
$
  \HL(y^{\alpha},\C_\Y) = \left\{ w \in H^1(y^\alpha,\C_\Y): w = 0 \text{ on }
    \partial_L \C_\Y \cup \Omega \times \{ \Y\} \right\},
$
and
\begin{equation}
\label{def:a_Y}
a_\Y: \HL(y^{\alpha},\C_\Y) \times \HL(y^{\alpha},\C_\Y) \rightarrow \mathbb{R},
\quad
a_\Y(w,\phi) = \frac{1}{d_s} \int_{\C_\Y} y^{\alpha}  \nabla w \cdot \nabla \phi \diff x.
\end{equation}
This problem admits a unique optimal solution $(\bar v, \orsf) \in \HL(y^{\alpha},\C_{\Y}) \times \Zad$. In addition, such a pair is optimal if and only if $\bar v$ solves \eqref{eq:alpha_harm_truncated} and 
\begin{equation}
\label{eq:VI}
 (\tr \bar{p} + \sigma \orsf + \nu \bar t, \rsf - \orsf )_{L^2(\Omega)} \geq 0 \quad \forall \rsf \in \Zad,
\end{equation}
where $\bar t \in \partial \psi (\orsf)$ and $\bar{p} \in \HL(y^{\alpha},\C_\Y)$ solves the truncated adjoint problem
\begin{equation}
\label{eq:p_truncated}
a_\Y(\phi,\bar{p}) = (  \tr \bar{v} - \usfd, \tr \phi )_{L^2(\Omega)} \quad \forall \phi \in \HL(y^{\alpha},\C_\Y).  
\end{equation}
The convex and Lipschitz function $\psi$ is defined as follows:
\begin{equation}
 \label{eq:psi}
 \psi: L^1(\Omega) \rightarrow \mathbb{R}, \quad \psi( \rsf ):= \int_{\Omega} |\rsf(x')| \diff x'.
\end{equation}

The following approximation result shows how $(\bar v, \orsf)$ approximates $(\bar \ue, \ozsf)$
\begin{proposition}[exponential convergence]
Let $(\oue,\ozsf) \in \HL(y^{\alpha},\C) \times \Zad$ and $(\bar v, \orsf) \in \HL(y^{\alpha},\C_{\Y}) \times \Zad$ be the solutions to the extended and truncated optimal control problems, respectively. Then,
\begin{align*}
% \label{eq:v-v^Y}
  \| \bar \zsf - \bar \rsf \|_{L^2(\Omega)} & \lesssim e^{-\sqrt{\lambda_1} \Y/4} \left(\| \bar{\rsf} \|_{L^2(\Omega)} + \| \usf_d \|_{L^2(\Omega)} \right),\\
% \label{eq:trv-v^Y}
  \| \nabla \left( \bar \ue  - \bar{v}  \right) \|_{L^2(y^{\alpha},\C)} & \lesssim e^{-\sqrt{\lambda_1} \Y/4} \left(\| \bar{\rsf} \|_{L^2(\Omega)} + \| \usf_d \|_{L^2(\Omega)} \right),
\end{align*}
where $\lambda_1$ denotes the first eigenvalue of $-\Delta$.
\end{proposition}
\begin{proof}
See \cite[Theorem 5.2]{OS:Sparse}. 
\end{proof}

\EO{To present the following result we introduce, for $\mathfrak{a}, \mathfrak{b} \in \mathbb{R}$, the  projection operator}
\begin{equation} 
\label{def:Pi}
 \Pi_{[ \mathfrak{a}, \mathfrak{b} ]}: L^2(\Omega) \rightarrow \Zad, \quad \Pi_{[ \mathfrak{a}, \mathfrak{b} ]}(x') = \min\{\mathfrak{b}, \max\{\mathfrak{a}, x'\}\}.
\end{equation}

\begin{proposition}[projection formulas]
If $\bar \rsf$, $\bar v$, $\bar p$, and $\bar t$ denote the optimal variables associated to the truncated optimal control problem, then
\begin{align} 
  \bar \rsf(x') & = \Pi_{[\asf,\bsf]}\left( -\frac{1}{\sigma} \left( \tr \bar p(x') + \nu \bar t(x') \right) \right),
  \label{eq:projection_r}
  \\
  \bar \rsf(x') &= 0 \quad \Leftrightarrow \quad |\tr \bar p (x')| \leq \nu,
  \label{eq:sparsity_r}
  \\
  \bar t(x') & = \Pi_{[-1,1]}\left( -\frac{1}{\nu}  \tr \bar p(x')  \right).
  \label{eq:projection_t}
\end{align}
\end{proposition}
\begin{proof}
See \cite[Corollary 3.7]{OS:Sparse}.
\end{proof}

%%%%%%%%%%%%%%%%%%%%%%%%%%%%%%%%%%%%%%%%%%%%%%%%%%%%%%%%%%%%%%%%%%%%%%%%%%%%%%%%%%%%%%%%%%%%%%%%%%%%%%%%%%%%%%%%%%%%%%%%%%%%%%%%%%%%%%%%%%%%%%%%%%%%%%%%%%%%%%%%%%%%%%%%%%
\subsection{A fully discrete scheme for the fractional optimal control problem}
\label{subsec:fully}
%%%%%%%%%%%%%%%%%%%%%%%%%%%%%%%%%%%%%%%%%%%%%%%%%%%%%%%%%%%%%%%%%%%%%%%%%%%%%%%%%%%%%%%%%%%%%%%%%%%%%%%%%%%%%%%%%%%%%%%%%%%%%%%%%%%%%%%%%%%%%%%%%%%%%%%%%%%%%%%%%%%%%%%%%%

In what follows we briefly recall the fully discrete scheme proposed in \cite{OS:Sparse} and review its a priori error analysis. To accomplish this task, we will assume in this section that
\begin{equation}
\label{eq:Omega_regular}
 \| w \|_{H^2(\Omega)} \lesssim \| \Delta_{x'} w \|_{L^2(\Omega)} \quad \forall w \in H^2(\Omega) \cap H^1_0(\Omega).
\end{equation}
This regularity assumption holds if, for instance, $\Omega$ is convex \cite{Grisvard}.

Before describing the aforementioned solution technique, we briefly recall the finite element approximation of \cite{NOS} for the state equation \eqref{eq:weak_alpha_harm}. Let $\T_{\Omega} = \{ K \}$ be a conforming and shape regular mesh of $\Omega$ into cells $K$ that are isoparametrically equivalent either to the unit cube $[0,1]^n$ or the unit simplex in $\R^n$ 
%We denote by $\Tr_{\Omega}$ the collections of all conforming refinements of an original mesh $\T_{\Omega}^0$. We assume that $\Tr_{\Omega}$ is shape regular 
\cite{BrennerScott,CiarletBook,Guermond-Ern}. 
Let $\mathcal{I}_{\Y}$ be a partition of $[0,\Y]$ with mesh points
\begin{equation}
\label{eq:graded_mesh}
  y_\ell = \left( \frac{\ell}{M}\right)^{\gamma} \Y, \quad \gamma > \frac{3}{(1-\alpha)}=\frac{3}{2s} > 1,\quad  \ell=0,\dots,M.
\end{equation}
We construct a mesh $\T_{\Y}$ over the cylinder $\C_{\Y}$ as $\T_{\Y} = \T_{\Omega} \otimes I_{\Y}$, the tensor product triangulation of $\T_{\Omega}$ and $\mathcal{I}_{\Y}$.
% We notice that each discretization of $\C_{\Y}$ depends on the truncation parameter $\Y$. 
The set of all the obtained meshes is denoted by $\Tr$. Notice that, owing to \eqref{eq:graded_mesh}, the meshes $\T_{\Y}$ are not shape regular but satisfy: if $T_1 = K_1 \times I_1$ and $T_2=K_2\times I_2$ are neighbors, \EO{then there is $\mu>0$ such that
$
h_{I_1} h_{I_2}^{-1} \leq \mu,
$
where $h_I = |I|$.}
%
%The following weak shape regularity condition is valid \cite{DL:05,NOS,NOS2}: there is a constant $\mu$ such that, for all $\T_{\Y} \in \Tr$, if $T_1 = K_1 \times I_1$ and $T_2=K_2\times I_2 \in \T_\Y$ have nonempty intersection, then 
%$
%     h_{I_1} h_{I_2}^{-1} \leq \mu,
%$
%where $h_I = |I|$. 
This condition allows for anisotropy in the extended variable $y$ \cite{DL:05,NOS,NOS2}, which is needed to compensate the rather singular behavior of $\ue$, solution to \eqref{eq:alpha_harm_weak}. We refer the reader to \cite{NOS} for details.

With the mesh $\T_{\Y} \in \Tr$ at hand, we define the finite element space
%the finite element space 
\begin{equation}
\label{eq:FESpace}
  \V(\T_\Y) = \left\{
            W \in C^0( \overline{\C_\Y} ): W|_T \in \mathcal{P}_1(K) \otimes \mathbb{P}_1(I) \ \forall T \in \T_\Y, \
            W|_{\Gamma_D} = 0
          \right\},
\end{equation}
where $\Gamma_D = \partial_L \C_{\Y} \cup \Omega \times \{ \Y\}$ is the Dirichlet boundary. If the base $K$ of the element $T = K \times I$ is a cube, $\mathcal{P}_1(K)$ stand for $\mathbb{Q}_1(K)$ -- the space of polynomials of degree not larger that one in each variable.
When $K$ is a simplex, the space $\mathcal{P}_1(K)$ is $\mathbb{P}_1(K)$, \ie the set
of polynomials of degree at most one.
%-- the space of polynomials of degree at most $1$
%. 
We also define $\U(\T_{\Omega})=\tr \V(\T_{\Y})$. Notice that $\U(\T_{\Omega})$ corresponds to a $\mathcal{P}_1$ finite element space over $\T_\Omega$. 
% 
% Before describing the numerical scheme introduced and developed in \cite{MR3429730}, we recall the regularity properties of the extended and truncated optimal controls $\ozsf$ and $\orsf$, respectively. If $\usf_{d} \in \mathbb{H}^{1-s}(\Omega)$ and $\asf \leq 0 \leq \bsf$ for $s \in (0,\tfrac{1}{2}]$, then $\ozsf \in H^1(\Omega) \cap \mathbb{H}^{1-s}(\Omega)$ \cite[Lemmas 3.5 and 5.9]{MR3429730}. Under the same framework, we have the same result for the truncated optimal control: $\orsf \in H^1(\Omega) \cap \mathbb{H}^{1-s}(\Omega)$ \cite[Lemma 5.9]{MR3429730}.

We now describe the \emph{fully discrete optimal control problem}. To accomplish this task, we first introduce the discrete sets
\begin{equation*}
\mathbb{Z}_{ad}(\T_{\Omega}) = \Zad \cap \mathbb{P}_0(\T_{\Omega}),
\quad 
\mathbb{P}_0(\T_{\Omega}) = \left\{Z \in L^{\infty}( \Omega ): Z|_K \in \mathbb{P}_0(K) \quad \forall K \in \T_\Omega \right\}.
\end{equation*}
%\ie it corresponds to the space of piecewise constant functions defined on the partition $\T_{\Omega}$ that satisfies the control bounds.

The fully discrete optimal control problem thus reads as follows: Find 
\[
\min \{ J(\tr V , Z): V \in \V(\T_{\Y}), Z \in \mathbb{Z}_{ad}(\T_{\Omega}) \}
\]
subject to
% the discrete state equation
\begin{equation}
\label{def:a_discrete}
a_\Y(V,W) =  ( Z, \tr W )_{L^2(\Omega)} \quad \forall W \in \V(\T_{\Y}).
\end{equation}
%%and the discrete control constraints
%%$
%%Z \in \mathbb{Z}_{ad}(\T_{\Omega}).
%%$
We recall that $J$ and $a_{\Y}$ are defined by \eqref{def:J} and \eqref{def:a_Y}, respectively. Standard arguments guarantee the existence of a unique optimal pair $(\bar{V}, \bar{Z}) \in \V(\T_\Y) \times \mathbb{Z}_{ad}(\T_{\Omega})$. In view of the results of \cite{NOS,MR3259962}, we invoke the discrete solution $\bar{V} \in \V(\T_\Y)$ and set
\begin{equation}
\label{eq:U_discrete}
\bar{U}:= \tr \bar{V}.
\end{equation}
We have thus obtained a fully discrete approximation $(\bar{U},\bar{Z}) \in  \U(\T_{\Omega}) \times \mathbb{Z}_{ad}(\T_{\Omega})$ of $(\ousf,\ozsf) \in \Hs \times \Zad$, the solution to the fractional optimal control problem.

The optimality conditions for the fully discrete optimal control problem read: the pair $(\bar{V}, \bar{Z}) \in \V(\T_\Y) \times \mathbb{Z}_{ad}(\T_{\Omega})$ is optimal if and only if $\bar V \in \V(\T_\Y)$ solves \eqref{def:a_discrete} and 
\begin{equation}
\label{eq:VI_discrete}
(\tr \bar{P} + \sigma \bar{Z} + \nu \bar \Lambda, Z- \bar{Z})_{L^2(\Omega)} \geq 0 \quad \forall Z \in \mathbb{Z}_{ad}(\T_{\Omega}),
\end{equation}
where $\bar \Lambda \in \partial \psi (\bar Z)$ and the optimal discrete adjoint state $\bar{P} \in \V(\T_\Y)$ solves
\begin{equation}
\label{eq:P_discrete}
a_\Y(W,\bar{P}) = ( \tr \bar{V} - \usfd , \textrm{tr}_{\Omega} W )_{L^2(\Omega)} \quad \forall W \in \V(\T_{\Y}).
\end{equation}

To write an priori error estimates for the aforementioned scheme, we first observe that $\#\T_{\Y} = M \, \# \T_\Omega$, and that $\# \T_\Omega \approx M^n$. Consequently, $\#\T_\Y \approx M^{n+1}$. Thus, if $\T_\Omega$ is quasi--uniform, we have that $h_{\T_{\Omega}} \approx (\# \T_{\Omega})^{-1/n}$.

\begin{theorem}[fractional control problem: error estimate]
\label{th:error_estimate}
Let $(\bar{V},\bar{Z}) $ $\in \V(\T_\Y) \times \mathbb{Z}_{ad}(\T_{\Omega})$ be the optimal pair for the fully discrete optimal control problem. Let $\bar{U} \in \U(\T_{\Omega})$ be defined as in \eqref{eq:U_discrete}. If $\Omega$ verifies \eqref{eq:Omega_regular} and $\usf_d \in \Ws$, then 
\begin{equation}
\label{fd2}
  \| \ozsf - \bar{Z} \|_{L^2(\Omega)} \lesssim  |\log (\# \T_{\Y})|^{2s}(\# \T_{\Y})^{\frac{-1}{n+1}} 
\left( \| \orsf \|_{H^1(\Omega)} + \| \usfd \|_{\Ws} \right),
\end{equation}
and
\begin{equation}
\label{fd1}
\| \ousf - \bar{U} \|_{\Hs} \lesssim  |\log (\# \T_{\Y})|^{2s}(\# \T_{\Y})^{\frac{-1}{n+1}} 
\left( \| \orsf \|_{H^1(\Omega)} + \| \usfd \|_{\Ws}  \right),
\end{equation} 
where the truncation parameter $\Y$, in the truncated optimal control problem, is chosen such that $\Y \approx \log( \# \T_{\Y} )$. The hidden constants in both inequalities are independent of the discretization parameters and the continuous and discrete optimal variables.
\end{theorem}
\begin{proof}
 See \cite[Theorem 6.4]{OS:Sparse}.
\end{proof}

\begin{remark}[conditions for a priori theory]
\rm
\EO{We stress that the results of Theorem \ref{th:error_estimate} are valid if $\Omega$ satisfies \eqref{eq:Omega_regular} and $\usf_d \in \Ws$. These conditions guarantee that $\ozsf$ and $\orsf$ belong to $H_0^1(\Omega)$; see \cite[Theorem 3.9]{OS:Sparse}.}
\end{remark}

We conclude this section by defining the following auxiliary variables that will be of importance in the a posteriori error analysis that we will perform:
\begin{equation}
\label{eq:v_star}
\mathpzc{v} \in \HL(y^{\alpha},\C_{\Y}): \quad a_{\Y} (\mathpzc{v},\phi) = (  \bar{Z}, \tr \phi )_{L^2(\Omega)} \quad \forall \phi \in \HL(y^{\alpha},\C_{\Y}),
\end{equation}
and
\begin{equation}
\label{eq:q_truncated}
\mathpzc{p} \in \HL(y^{\alpha},\C_{\Y}):\quad a_\Y(\phi,\mathpzc{p}) = (  \tr \bar{V} - \usfd, \tr \phi )_{L^2(\Omega)} \quad \forall \phi \in \HL(y^{\alpha},\C_{\Y}).
\end{equation}

%%%%%%%%%%%%%%%%%%%%%%%%%%%%%%%%%%%%%%%%%%%%%%%%%%%%%%%%%%%%%%%%%%%%%%%%%%%%%%%%%%%%%%%%%%%%%%%%%%%%%%%%%%%%%%%%%%%%%%%%%%%%%%%%%%%%%%%%%%%%%%%%%%%%%%%%%%%%%%%%%%%%%%%%%%
\section{An \emph{ideal} a posteriori error estimator}
\label{subsec:ideal_a_posteriori}

The main goal of this work is the derivation and analysis of a computable a posteriori error estimator for  problem \eqref{def:minJ}--\eqref{eq:control_constraints}. An a posteriori error estimator is a computable quantity that provides information about the local quality of the underlying approximated solution. It is an essential ingredient of AFEMs,
which are iterative methods that improve the quality of the approximated solution and are based on loops of the form
\begin{equation}
\label{eq:loop}
\textsf{\textup{SOLVE}} \rightarrow \textsf{\textup{ESTIMATE}} \rightarrow \textsf{\textup{MARK}} \rightarrow \textsf{\textup{REFINE}}.
\end{equation}
A posteriori error estimators are the heart of the step $\textsf{\textup{ESTIMATE}}$. The theory for linear and second--order elliptic boundary value problems is well--established. We refer the reader to \cite{MR1770058,NSV:09,Verfurth} for an up to date discussion that also includes the design of AFEMs, convergence results, and optimal complexity.

The a posteriori error analysis for finite element approximations of constrained optimal control problems is currently under development; the main source of difficulty being its inherent nonlinear feature. Starting with the pioneering work \cite{MR1887737}, several authors have contributed to its advancement. For an up to date survey on a posteriori error analysis for optimal control problems we refer the reader to \cite{AOSR,KRS,MR3485971}. In contrast to these advances, the theory for optimal control problems involving a sparsity functional, as \eqref{def:minJ}, is much less developed. To the best of our knowledge the only works that provides an advance concerning this matter are \cite{AFO:17} and \cite{MR2826983}. In \cite{MR2826983}, the authors propose a residual--type a posteriori error estimator, for the picewise constant discretization of the optimal control, and prove that it yields an upper bound for the approximation error of the state and control variables \cite[Theorem 6.2]{MR2826983}. These results have been recently extended in \cite{AFO:17}, \EO{where the authors consider three different strategies
to approximate the control variable: piecewise constant discretization, piecewise linear discretization, and the so--called variational discretization approach. The authors propose a posteriori error estimators, for each scheme, and derive reliability and efficiency estimates.}

In this work, we follow \cite{HHIK,KRS,MR1887737} and design an a posteriori error estimator for problem \eqref{def:minJ}--\eqref{eq:control_constraints}. To accomplish this task, it is essential to have at hand an error estimator for \eqref{eq:alpha_harm_truncated}. The latter equation involves a nonuniformly elliptic operator with the variable coefficient $y^{\alpha}$ that vanishes ($0 < \alpha < 1$) or blows up $(-1 < \alpha < 0)$ as $y \downarrow 0$. Consequently, the design of error estimators for \eqref{eq:alpha_harm_truncated} is far from being trivial: In fact, a simple computation reveals that the usual residual estimator does not apply. In addition, numerical evidence shows that anisotropic refinement in the extended dimension is essential to observe optimal rates of convergence. Inspired by \cite{MR880421,CF:00,MNS02}, the authors of \cite{CNOS2} analyze an a posteriori error estimator for \eqref{eq:alpha_harm_truncated}  based on the solution of weighted local problems on cylindrical stars.

As a first step, we explore an \emph{ideal} anisotropic estimator that can be decomposed as the sum of four contributions:
\begin{equation}
\label{eq:ideal_aux}
 \E^2_{\textrm{ocp}}= \E_{V}^2 + \E_{P}^2 + \E^2_{Z} + \E^2_{\Lambda}.
\end{equation}
The error indicators $\E_{V}$ and $\E_{P}$, that are related to discretization of the state and adjoint equations, follow from \cite{CNOS2}. They allow for the nonuniform coefficient $y^{\alpha}$ and the anisotropic meshes in the family $\Tr$. The error indicators $\E_{Z}$ and $\E_{\Lambda}$ are related to the discretization of the control variable and the associated subgradient. We refer to this estimator as \emph{ideal} since the computation of $\E_{V}$ and $\E_{P}$ involve the resolution of problems in infinite dimensional spaces.  We derive, in Section \ref{subsec:ideal_a_posteriori}, the equivalence between the ideal estimator \eqref{eq:ideal_aux} and the error without oscillation terms. Such an equivalence relies on a geometric condition imposed on the mesh that is independent of the exact optimal variables and is computationally implementable. This ideal estimator sets the basis to define, in Section \ref{subsec:computable_a_posteriori}, a computable one, which is also decomposed as the sum of four contributions. This computable estimator is, under certain assumptions, equivalent, up to data oscillations terms, to the error. 

%%%%%%%%%%%%%%%%%%%%%%%%%%%%%%%%%%%%%%%%%%%%%%%%%%%%%%%%%%%%%%%%%%%%%%%%%%%%%%%%%%%%%%%%%%%%%%%%%%%%%%%%%%%%%%%%%%%%%%%%%%%%%%%%%%%%%%%%%%%%%%%%%%%%%%%%%%%%%%%%%%%%%%%%%%%
\subsection{Preliminaries}
\label{subsec:preliminaries}
%%%%%%%%%%%%%%%%%%%%%%%%%%%%%%%%%%%%%%%%%%%%%%%%%%%%%%%%%%%%%%%%%%%%%%%%%%%%%%%%%%%%%%%%%%%%%%%%%%%%%%%%%%%%%%%%%%%%%%%%%%%%%%%%%%%%%%%%%%%%%%%%%%%%%%%%%%%%%%%%%%%%%%%%%%%

We follow \cite[Section 5.1]{CNOS2} and introduce some notation and terminology. Given a node $z$ on $\T_{\Y}$, we write $z = (z',z'')$ where $z'$ and $z''$ correspond to nodes on $\T_{\Omega}$ and $\mathcal{I}_{\Y}$, respectively. 

Given $K \in \T_{\Omega}$, we denote by $\N(K)$ and $\Nin(K)$ the set of nodes and interior nodes of $K$, respectively. We also define 
\[
 \N(\T_{\Omega}) = \cup \{ \N(K): K \in \T_\Omega \}, \qquad \Nin(\T_{\Omega}) = \cup \{\Nin(K): K \in \T_\Omega \}.
\]
Given $T \in \T_{\Y}$, we define $\N(T)$, $\Nin(T)$, $\N(\T_{\Y})$, and $\Nin(\T_{\Y})$, accordingly. 

Given $z' \in \N(\T_{\Omega})$, we define 
$
 S_{z'} = \cup \{  K \in \T_\Omega : \ K \ni z  \}\subset \Omega 
$
%%\[
%%  S_{z'} := \bigcup \left\{ K \in \T_\Omega : \ K \ni z' \right\} \subset \Omega 
%%\]
and the \emph{cylindrical star} around $z'$ as
\begin{equation}
\label{def:cylindrical_star}
  \C_{z'} := \bigcup\left\{ T \in \T_\Y : T = K \times I,\ K \ni z'  \right\}= S_{z'} \times (0,\Y) \subset \C_{\Y}.
\end{equation}

Given a cell $K \in \T_{\Omega}$, we define the \emph{patch} $S_K$ as 
$
  S_K := \bigcup_{z' \in K} S_{z'}.
$
For $T \in \T_\Y$, we define the patch $S_T$ similarly. Given $z' \in \N(\T_{\Omega})$
we define its \emph{cylindrical patch} as
\[
\D_{z'} := \bigcup  \left\{ \C_{w'}: w' \in S_{z'} \right\} \subset \C_{\Y}.
\]

Finally, for each $z' \in \N(\T_{\Omega})$, we set $h_{z'} := \min\{h_{K}: K \ni z' \}$. 
% 
% If $\{ \phi_z: z \in \Nin(\T_{\Y}) \}$ denotes the canonical basis of $\V(\T_{\Y})$ and $W \in \V(\T_{\Y})$, then
% \[
%  W = \sum_{z \in \Ninn(\T_{\Y})} W(z) \phi_z.
% \]
% Analogously, we denote by $\{ \varphi_{z'}: z' \in \Nin(\T_{\Omega}) \}$ the the canonical basis of the discrete space $\U(\T_{\Omega}) = \tr \V(\T_{\Y})$.

%%%%%%%%%%%%%%%%%%%%%%%%%%%%%%%%%%%%%%%%%%%%%%%%%%%%%%%%%%%%%%%%%%%%%%%%%%%%%%%%%%%%%%%%%%%%%%%%%%%%%%%%%%%%%%%%%%%%%%%%%%%%%%%%%%%%%%%%%%%%%%%%%%%%%%%%%%%%%%%%%%%%%%%%%%%
\subsection{Local weighted Sobolev spaces}
\label{subsec:local_spaces}
%%%%%%%%%%%%%%%%%%%%%%%%%%%%%%%%%%%%%%%%%%%%%%%%%%%%%%%%%%%%%%%%%%%%%%%%%%%%%%%%%%%%%%%%%%%%%%%%%%%%%%%%%%%%%%%%%%%%%%%%%%%%%%%%%%%%%%%%%%%%%%%%%%%%%%%%%%%%%%%%%%%%%%%%%%%

%In what follows 
We define local weighted Sobolev spaces that will be instrumental for our analysis.

\begin{definition}[local weighted Sobolev spaces]
Given $z' \in \N(\T_\Omega)$, we define
\begin{equation}
\label{eq:local_space}
\W(\C_{z'}) = \left \{ w \in H^1(y^{\alpha},\C_{z'} ): w = 0 \textrm{ on } \partial \C_{z'} 
\setminus \Omega \times \{ 0\} \right \},
\end{equation}
where $\C_{z'}$ denotes the cylindrical star around $z'$ defined in \eqref{def:cylindrical_star}.
\end{definition}

Since $y^\alpha$ belongs to the class $A_2(\R^{n+1})$, the space $\W(\C_{z'})$ is Hilbert \cite{Javier,FKS:82,GU,Muckenhoupt,Turesson}. In addition, 
%%\cite[Proposition 5.8]{CNOS2} reveals that
%%\begin{equation}
%%\label{eq:Poincare}
%%  \| w \|_{L^2(y^{\alpha},\C_{z'})} \lesssim \Y \|  \nabla w \|_{L^2(y^{\alpha},\C_{z'})} \quad \forall w \in \W(\C_{z'}),
%%\end{equation}
%%with $\Y$ corresponding to the truncation parameter defined in Section \ref{subsec:truncated}. We also 
we have that \cite[Proposition 2.1]{CDDS:11}
\begin{equation}
 \label{Trace_estimate_local}
 \| \tr w \|_{L^2(S_{z'})} \leq C_{\tr} \| \nabla w \|_{L^2(y^{\alpha},\C_{z'})} \quad \forall w \in \W(\C_{z'}),
\quad  C_{\tr} \leq d_s^{-\frac{1}{2}}.
\end{equation}
%%where $C_{\tr} \leq d_s^{-\frac{1}{2}}$.
%and notice that the same arguments of \cite[Section 2.3]{CNOS2} yield $C_{\tr} \leq d_s^{-\frac{1}{2}}$.

\subsection{Auxiliary variables}
\label{subsec:auxiliary_variables}

\EO{To perform an a posteriori error analysis, we introduce the following two auxiliary variables:}
%In the next section we will propose an error estimator whose analysis relies on two 
%auxiliary variables: $\tilde \rsf$ and $\tilde \lambda$. With the operator $\Pi_{[\mathfrak{a},\mathfrak{b}]}$, defined as in \eqref{def:Pi}, at hand, we define
\begin{equation}
\label{eq:auxiliary_variables}
\tilde \lambda(x') := \Pi_{[-1,1]} \left( - \frac{1}{\nu} \tr \bar  P(x') \right), \qquad 
\tilde \rsf := \Pi_{[\asf,\bsf]} \left( - \frac{1}{\sigma} \left( \tr \bar  P(x') + \nu \tilde \lambda(x') \right) \right),
\end{equation}
\EO{where $\Pi_{[\mathfrak{a},\mathfrak{b}]}$ is defined as in \eqref{def:Pi}.}

%In the next result 
We now derive two important properties that will be essential for our analysis.
%the a posteriori error analysis that we will perform. 

\begin{lemma}
Let $\tilde \rsf \in \Zad$ and $\tilde \lambda \in L^{\infty}(\Omega)$ be defined as in \eqref{eq:auxiliary_variables}. Then, $\tilde \rsf$ can be characterized by the variational inequality 
\begin{equation}
 \label{eq:key_VI}
 (\tr \bar P + \sigma \tilde \rsf + \nu \tilde \lambda, \rsf - \tilde \rsf)_{L^2(\Omega)} \geq 0 \quad \forall \rsf \in \Zad,
\end{equation}
and $\tilde \lambda \in \partial \psi(\tilde \rsf)$.
\label{lemma:key_results}
\end{lemma}
\begin{proof}
The variational inequality \eqref{eq:key_VI} follows immediately from the arguments elaborated in the proof of \cite[Lemma 2.26]{Tbook}. It thus suffices to prove that $\tilde \lambda \in \partial \psi(\tilde \rsf)$. To accomplish this task, we first assume that $\tilde \rsf (x') = 0$. In view of the definition of $\tilde \lambda$, we immediately conclude that $\tilde \lambda(x') \in [-1,1]$. Let us now assume that $\tilde \rsf(x') > 0$. This and \eqref{eq:auxiliary_variables} imply that
\[
 \tr \bar P(x') + \nu \EO{\tilde \lambda(x')} < 0 \iff \tilde \lambda(x') < -\frac{1}{\nu}\tr \bar P(x') \implies \tilde \lambda(x') = 1.
\]
Analogously, we can prove that, if $\tilde \rsf(x') < 0$, then $\tilde \lambda(x') = -1$. In conclusion, $\tilde \lambda \in L^{\infty}(\Omega)$ is such that
\[
\tilde \lambda(x') = 1, \quad \tilde \rsf (x') > 0, \qquad \tilde \lambda(x') = - 1, \quad \tilde \rsf (x') < 0, \quad \tilde \lambda(x') \in [-1,1], \quad \tilde \rsf (x') = 0.
\]
This, that is equivalent to $\tilde \lambda \in \partial \psi(\tilde \rsf)$, concludes the proof.
\end{proof}

\EO{We conclude the section by defining the following auxiliary variables:
\begin{equation}
\label{def:tildev}
\tilde{v} \in \HL(y^{\alpha},\C_{\Y}): \quad a_{\Y} (\tilde{v},\phi) = ( \tilde{\rsf}, \tr \phi )_{L^2(\Omega)} \quad \forall \phi \in \HL(y^{\alpha},\C_{\Y}),
\end{equation}
and}
\begin{equation}
\label{eq:w_truncated}
\tilde p \in \HL(y^{\alpha},\C_{\Y}): \quad a_{\Y} (\phi,\tilde p) = (  \tr \tilde{v} - \usfd, \tr \phi )_{L^2(\Omega)} \quad \forall \phi \in \HL(y^{\alpha},\C_{\Y}).
\end{equation}

%%%%%%%%%%%%%%%%%%%%%%%%%%%%%%%%%%%%%%%%%%%%%%%%%%%%%%%%%%%%%%%%%%%%%%%%%%%%%%%%%%%%%%%%%%%%%%%%%%%%%%%%%%%%%%%%%%%%%%%%%%%%%%%%%%%%%%%%%%%%%%%%%%%%%%%%%%%%%%%%%%%%%%%%%%%
\subsection{A posteriori error analysis}
\label{subsec:ideal_a_posteriori_analysis}
%%%%%%%%%%%%%%%%%%%%%%%%%%%%%%%%%%%%%%%%%%%%%%%%%%%%%%%%%%%%%%%%%%%%%%%%%%%%%%%%%%%%%%%%%%%%%%%%%%%%%%%%%%%%%%%%%%%%%%%%%%%%%%%%%%%%%%%%%%%%%%%%%%%%%%%%%%%%%%%%%%%%%%%%%%%

In this section we design and study an \emph{ideal} a posteriori error estimator for  \eqref{def:minJ}--\eqref{eq:control_constraints}. We refer to such an estimator as ideal since it involves the resolution of local problems on the infinite dimensional spaces $\W(\C_{z'})$; the estimator is therefore not computable. We prove that it is equivalent to the error without oscillation terms; see Theorems \ref{thm:ideal_1} and \ref{thm:ideal_2} below. 

We define the ideal a posteriori error estimator as the sum of four contributions:
\begin{multline}
 \label{eq:defofEocp}
 \E^2_{\textrm{ocp}}(\bar{V},\bar{P},\bar{Z}, \bar{\Lambda}; \T_{\Y}) = \E^2_{V}(\bar{V},\bar{Z}; \N(\T_{\Omega})) + \E^2_{P}(\bar{P},\bar{V}; \N(\T_{\Omega})) \\
 + \E^2_{Z}(\bar{Z},\bar{P}; \T_{\Omega}) + \E^2_{\Lambda}(\bar{\Lambda},\bar{P}; \T_{\Omega}),
\end{multline}
where $\bar{V}$, $\bar{P}$, $\bar{Z}$, and $\bar \Lambda$ denote the optimal variables associated to the fully discrete optimal control problem of Section \ref{subsec:fully}. In what follows, we describe each contribution in \eqref{eq:defofEocp} separately. To accomplish this task, we define the bilinear form
\begin{equation}
 \label{eq:a_local}
  a_{z'}: \W(\C_{z'}) \times \W(\C_{z'}) \rightarrow \mathbb{R}, \quad a_{z'}(w,\phi) = \frac{1}{d_s} \int_{\C_{z'}} y^{\alpha} \nabla w \cdot \nabla \phi \diff x.
\end{equation}

The first contribution in \eqref{eq:defofEocp} corresponds to the a posteriori error estimator of \cite[Section 5.3]{CNOS2}. Let us define, for $z' \in \N(\T_{\Omega})$,
\begin{equation}
\label{eq:ideal_local_problemV}
\zeta_{z'} \in \W(\C_{z'}): \quad a_{z'}(\zeta_{z'},\psi) = \langle \bar{Z}, \tr \psi  \rangle  -   a_{z'}(\bar{V}, \psi) \quad \forall \psi \in \W(\C_{z'}).
\end{equation}
With this definition at hand, we define the posteriori error indicators and estimator
\begin{equation}
\label{eq:defofEV} 
\E_{V}^2(\bar{V},\bar{Z}; \C_{z'}) = \| \nabla\zeta_{z'} \|^2_{L^2(y^{\alpha},\C_{z'})}, 
\, \, 
\E_{V}^2(\bar{V},\bar{Z}; \N(\T_{\Omega})) =  \sum_{z' \in \N(\T_{\Omega})} \E_{V}^2(\bar{V},\bar{Z}; \C_{z'}).
\end{equation}

We now describe the second contribution in \eqref{eq:defofEocp}. Let us define, for $z' \in \N(\T_{\Omega})$,
\begin{equation}
\label{eq:ideal_local_problemP}
\chi_{z'} \in \W(\C_{z'}): \quad 
a_{z'}(\chi_{z'},\psi) = \langle \tr \bar{V} - \usf_d, \tr \psi  \rangle  - a_{z'}(\psi,\bar{P}) \quad \forall \psi \in \W(\C_{z'}).
\end{equation}
We define the posteriori error indicators and estimator
\begin{equation}
\label{eq:defofEP} 
\E_{P}^2(\bar{P},\bar{V}; \C_{z'}) = \| \nabla\chi_{z'} \|^2_{L^2(y^{\alpha},\C_{z'})},
\, 
\E_{P}^2(\bar{P},\bar{V}; \N( \T_{\Omega} ) ) = \sum_{z' \in \N( \T_{\Omega} )} \E_{P}^2(\bar{P},\bar{V}; \C_{z'}).
\end{equation}

The third contribution in \eqref{eq:defofEocp} is defined as follows:
\begin{equation}
\label{eq:defofEZglobal}
\E_{Z}(\bar{Z},\bar{P}; K) = \| \bar{Z} - \tilde \rsf \|_{L^2(K)},
\quad  
\E_{Z}^2(\bar{Z},\bar{P}; \T_{\Omega}) = \sum_{K \in \T_{\Omega}}  \E^2_{Z}(\bar{Z},\bar{P}; K),
\end{equation}
where the auxiliary variable $\tilde \rsf$ is defined as in \eqref{eq:auxiliary_variables}.

Finally, and having in mind the definition of the auxiliary variable $\tilde \lambda$, given in \eqref{eq:auxiliary_variables}, we define the fourth contribution in \eqref{eq:defofEocp}:
\begin{equation}
\label{eq:defofELambdaglobal}
\E_{\Lambda}(\bar{\Lambda},\bar{P}; K) = \| \bar{\Lambda} - \tilde \lambda\|_{L^2(K)},
\quad  
\E_{\Lambda}^2(\bar{\Lambda},\bar{P}; \T_{\Omega}) = \sum_{K \in \T_{\Omega}}  \E^2_{\Lambda}(\bar{\Lambda},\bar{P}; K).
\end{equation}

Since $\bar V$ can be seen as the finite element approximation of $\mathpzc{v}$, defined in \eqref{eq:v_star}, within the space $\V(\T_{\Y})$, we have that \cite[Proposition 5.14]{CNOS2}: 
\begin{equation}
\label{eq:reliability_state_v}
 \| \nabla (\mathpzc{v} - \bar V) \|_{L^2(y^{\alpha},\C_{\Y})}  \lesssim \E_{V}(\bar{V},\bar{Z}; \N(\T_{\Omega})).
\end{equation}
Similarly,
\begin{equation}
\label{eq:reliability_state_p}
 \| \nabla (\mathpzc{p} - \bar P) \|_{L^2(y^{\alpha},\C_{\Y})}  \lesssim \E_{P}(\bar{P},\bar{V}; \N(\T_{\Omega})).
\end{equation}
These estimates are essential to prove the estimate \eqref{eq:control_error} below.

\EO{\begin{remark}[implementable geometric condition]\rm
It has been proven rather challenging to derive a posteriori error estimates on anisotropic meshes. To allow for graded meshes in $\Omega$ (needed to compensate geometric singularities and incompatible data) and anisotropic meshes in the $y$ variable, the results of \cite[Section 5.3]{CNOS2} assume the following implementable geometric condition
over $\Tr$: 
\begin{equation}
\label{condition}
\exists C_{\Tr} >0: \quad \T_{\Y} \in \Tr \implies
 h_{\Y} \leq C_{\Tr} h_{z'}  \quad \forall z' \in \T_{\Omega}.
\end{equation}
The term $h_{\Y}$ denotes the largest size among all the elements in the mesh $I_{\Y}$. This condition, \emph{that is fully implementable}, 
is specifically needed in the analysis that lead to \eqref{eq:reliability_state_v} and \eqref{eq:reliability_state_p}; more precisely in the inequality (5.18) of \cite{CNOS2}.
 \end{remark}
}

Define the errors $e_{V}: = \bar v - \bar V$, $e_{P}:= \bar p - \bar P$, $e_{Z} = \bar \rsf - \bar Z$, $e_{\Lambda} = \bar t - \bar \Lambda$, the vector $e = (e_V,e_P,e_{Z},e_{\Lambda})$, and the norm
\begin{equation}
 \label{eq:total_error}
\VERT e \VERT^2 :=  \| \nabla e_V \|^2_{L^2(y^{\alpha},\C_{\Y})}  +  \| \nabla e_P \|^2_{L^2(y^{\alpha},\C_{\Y})} + \| e_Z \|^2_{L^2(\Omega)} + \| e_{\Lambda} \|^2_{L^2(\Omega)}.
\end{equation}

\subsubsection{Reliability}

We obtain the global reliability of $\E_{\textrm{ocp}}$, defined in \eqref{eq:defofEocp}. The proof combines the arguments of \cite{AO:IMA} with elements of Sections \ref{subsec:subdifferentials} and \ref{subsec:auxiliary_variables}.

\begin{theorem}[global upper bound]
\label{thm:ideal_1}
Let $(\bar{v}, \bar{p}, \orsf, \bar t) \in \HL(y^{\alpha},\C_{\Y}) \times \HL(y^{\alpha},\C_{\Y}) \times \Zad \times L^{\infty}(\Omega)$ be the optimal variables associated to the truncated optimal control problem of Section \ref{subsec:truncated} and $(\bar{V},\bar{P},\bar{Z}, \bar \Lambda) \in \V(\T_{\Y}) \times \V(\T_{\Y}) \times \mathbb{Z}_{ad}(\T_{\Omega}) \times \mathbb{P}_0(\T_{\Omega})$ its numerical approximation as described in Section \ref{subsec:fully}. If \eqref{condition} holds, then
\begin{equation}
\label{eq:reliability}
 \VERT e \VERT \lesssim  \E_{\textrm{ocp}}(\bar{V},\bar{P},\bar{Z}, \bar{\Lambda}; \T_{\Y}),
\end{equation}
where the hidden constant is independent of the continuous and discrete optimal variables, the size of the elements in the meshes $\T_{\Omega}$ and $\T_{\Y}$, $\# \T_{\Omega}$, and $\# \T_{\Y}$.
\label{th:reliability}
\end{theorem}
\begin{proof} We proceed in five steps.
 
Step 1. Applying the triangle inequality we immediately arrive at
% 
% With the definition \eqref{eq:defofEZ} of the local error indicator $\E_{Z}$ in mind, we define the auxiliary control $\tilde{\rsf} = \Pi (-\frac{1}{\mu} \tr \bar{P})$ and notice that it verifies 
% \begin{equation}
% \label{eq:VI_rtilde}
%  ( \tr \bar{P} + \mu \tilde{\rsf} , \rsf - \tilde{\rsf} )_{L^2(\Omega)} \geq 0 \quad \forall \rsf \in \Zad.
% \end{equation}
% Then, an application of the triangle inequality yields
\begin{equation}
\label{eq:r-Z-rtilde}
 \| \orsf - \bar{Z} \|^2_{L^2(\Omega)} \leq  2\| \orsf - \tilde{\rsf} \|^2_{L^2(\Omega)}  +  2\| \tilde{\rsf}- \bar{Z} \|^2_{L^2(\Omega)} =  2\| \orsf - \tilde{\rsf} \|^2_{L^2(\Omega)} + 2 \E_{Z}^2,
\end{equation}
where $\E_{Z} = \E_{Z}(\bar Z, \bar P; \T_{\Omega})$ corresponds to the error estimator defined in \eqref{eq:defofEZglobal} and $\tilde \rsf$ denotes the auxiliary variable defined in \eqref{eq:auxiliary_variables}. 

Step 2. The previous estimate reveals that it thus suffices to bound $\| \orsf - \tilde{\rsf} \|_{L^2(\Omega)}$. To accomplish this task, we set $\rsf = \tilde \rsf$ in \eqref{eq:VI} and $\rsf = \orsf$ in \eqref{eq:key_VI} and add the obtained inequalities to arrive at
\begin{equation}
\label{eq:Step_2}
\sigma \| \orsf - \tilde{\rsf} \|^2_{L^2(\Omega)} \leq (\tr(\bar{p} - \bar{P}), \tilde{\rsf} - \orsf )_{L^2(\Omega)} + \nu (\bar t - \tilde \lambda, \tilde{\rsf} - \orsf )_{L^2(\Omega)},
\end{equation}
where $\tilde \lambda$ is defined in \eqref{eq:auxiliary_variables}, $\bar t \in \partial \psi(\orsf)$, and $\bar{p}$ and $\bar{P}$ solve \eqref{eq:p_truncated} and \eqref{eq:P_discrete}, respectively. The results of Lemma \ref{lemma:key_results} yield $\tilde \lambda \in \partial \psi(\tilde \rsf)$. This, in view of \eqref{eq:subdiff_monotone}, implies that
\[
 (\bar t - \tilde \lambda, \tilde{\rsf} - \orsf )_{L^2(\Omega)} \leq 0,
\]
and thus that
\begin{equation}
\label{eq:Step_22}
\sigma \| \orsf - \tilde{\rsf} \|^2_{L^2(\Omega)} \leq (\tr(\bar{p} - \bar{P}), \tilde{\rsf} - \orsf )_{L^2(\Omega)}.
\end{equation}

Step 3. 
\EO{To control the right hand side of \eqref{eq:Step_22}, we utilize the auxiliary adjoint state $\mathpzc{p}$, defined in \eqref{eq:q_truncated}, and write $\bar{p} - \bar{P} = (\bar{p}-\mathpzc{p}) + (\mathpzc{p} - \bar{P})$. Thus,
\begin{equation}
\label{eq:muleq}
\sigma \| \orsf - \tilde{\rsf} \|^2_{L^2(\Omega)} \leq  (\tr(\bar{p} - \mathpzc{p}), \tilde{\rsf} - \orsf )_{L^2(\Omega)} + (\tr(\mathpzc{p}- \bar{P}), \tilde{\rsf} - \orsf )_{L^2(\Omega)} =: \mathrm{I} + \mathrm{II}.
\end{equation}
We bound $\textrm{II}$ in view of the trace estimate \eqref{Trace_estimate}, \eqref{eq:reliability_state_p}, and Young's inequality:
\begin{multline}
|\textrm{II}|  \lesssim \| \nabla( \mathpzc{p}- \bar{P} ) \|_{L^2(y^{\alpha},\C_{\Y})} \| \tilde{\rsf} - \orsf \|_{L^2(\Omega)} \lesssim \E_{P}(\bar{P},\bar{V};\N(\T_{\Omega}))  \| \tilde{\rsf} - \orsf \|_{L^2(\Omega)}
\\
 \leq \frac{\sigma}{4}  \| \tilde{\rsf} - \orsf \|^2_{L^2(\Omega)} + C \E_{P}^2(\bar{P},\bar{V};\N(\T_{\Omega})),
\label{eq:estimate_for_I}
\end{multline}
where $C$ denotes a positive constant. To bound $\mathrm{I}$, we invoke the auxiliary state $\tilde p$, defined in \eqref{eq:w_truncated}, and write $\bar{p} - \mathpzc{p} = ( \bar{p} - \tilde p ) + (\tilde p - \mathpzc{p} )$. We estimate each contribution to $\mathrm{I}$ separately. First, we observe that
\[
a_{\Y} (\bar{v} - \tilde{v},\phi_v) = (\orsf- \tilde{\rsf}, \tr \phi_v )_{L^2(\Omega)}, \quad a_{\Y} (\phi_p,\bar{p} - \tilde p) = (  \tr (\bar{v}-\tilde{v}), \tr \phi_p )_{L^2(\Omega)}
\]
for all $\phi_v$ and $\phi_p$ in $\HL(y^{\alpha},\C_{\Y})$, respectively.  Set $\phi_v = \bar p - \tilde p$ and $\phi_p = \bar v - \tilde v$. Thus
\begin{equation}
\label{eq:estimate_for_II_1}
 \mathrm{I}_1:= (\tr(\bar{p} - \tilde p), \tilde{\rsf} - \orsf )_{L^2(\Omega)} = - a_{\Y}  (\bar{v} - \tilde{v},\bar p -\tilde p) = - \| \tr (\bar{v} - \tilde{v}) \|^2_{L^2(\Omega)} \leq 0.
 \end{equation}
We now control $\mathrm{I}_2:= (\tr(\tilde p - \mathpzc{p}), \tilde{\rsf} - \orsf )_{L^2(\Omega)}$, where $\tilde p$ and $ \mathpzc{p}$ are defined in \eqref{eq:w_truncated} and \eqref{eq:q_truncated}, respectively. To accomplish this task, we observe that
\[
\tilde p - \mathpzc{p}: \quad a_{\Y} (\phi,\tilde p - \mathpzc{p}) = (  \tr (\tilde{v}-\bar{V}), \tr \phi )_{L^2(\Omega)} \quad \forall \phi \in \HL(y^{\alpha},\C_{\Y}).
\]
Consequently, \eqref{Trace_estimate} and a stability estimate for the previous problem yield
\begin{equation}
\label{eq:aux1}
 |\mathrm{I}_2| \lesssim \| \nabla ( \tilde p - \mathpzc{p} )  \|_{L^2(y^{\alpha},\C_{\Y})} \|\tilde{\rsf} - \orsf  \|_{L^2(\Omega)} \lesssim \|  \tr (\tilde{v}-\bar{V})  \|_{L^2(\Omega)} \|\tilde{\rsf} - \orsf  \|_{L^2(\Omega)}.
\end{equation}
This, \eqref{eq:muleq}, \eqref{eq:estimate_for_I}, and \eqref{eq:estimate_for_II_1} allow us to conclude that
\begin{equation}
\label{eq:new_step}
\sigma \| \orsf - \tilde{\rsf} \|^2_{L^2(\Omega)} \leq \frac{\sigma}{2} \| \orsf - \tilde{\rsf} \|^2_{L^2(\Omega)} + C_1\E_{P}^2(\bar{P},\bar{V};\N(\T_{\Omega})) + C_2 \|  \tr (\tilde{v}-\bar{V})  \|^2_{L^2(\Omega)},
\end{equation}
where $C_1$ and $C_2$ are positive constants. The control of $ \| \tr (\tilde{v}-\bar{V})  \|_{L^2(\Omega)}$ follows from the estimates
$
\|  \tr (\tilde{v}-\mathpzc{v})  \|_{L^2(\Omega)} \lesssim \E_{Z}(\bar{Z},\bar{P};\T_{\Omega})
$
and 
\[
\|  \tr (\mathpzc{v}-\bar{V})  \|_{L^2(\Omega)} \lesssim \| \nabla( \mathpzc{v}-\bar{V} ) \|_{L^2(y^{\alpha},\C_{\Y})} \lesssim \E_{V}(\bar{V},\bar{Z};\N(\T_{\Omega})). 
\]
Replacing the previous two estimates into \eqref{eq:new_step} and the obtained one into \eqref{eq:r-Z-rtilde},
we arrive at
\[
\| e_Z \|^2_{L^2(\Omega)} \lesssim \E^2_{V}(\bar{V},\bar{Z};\N(\T_{\Omega})) + \E^2_{P}(\bar{P},\bar{V};\N(\T_{\Omega}))+ \E^2_{Z}(\bar{Z},\bar{P};\T_{\Omega}).
\]
Similar arguments can be applied to bound $\| \nabla e_V\|_{L^2(y^{\alpha},\C_{\Y})} $ and $ \| \nabla e_P \|_{L^2(y^{\alpha},\C_{\Y})}$. We can thus arrive at the estimate}
\begin{equation}
\| \nabla e_V\|_{L^2(y^{\alpha},\C_{\Y})} + \| \nabla e_P \|_{L^2(y^{\alpha},\C_{\Y})} + \| e_Z \|_{L^2(\Omega)}  
\lesssim \E_{\textrm{ocp}}(\bar{V},\bar{P},\bar{Z},\bar{\Lambda}; \T_{\Y}). 
\label{eq:control_error}
\end{equation}

Step 4. We now control the error in the approximation of the subgradient, \ie $\| \bar t - \bar \Lambda \|_{L^2(\Omega)}$. To accomplish this task, we first apply the triangle inequality and write
\[
 \| \bar t - \bar \Lambda \|_{L^2(\Omega)} \leq \| \bar t - \tilde \lambda \|_{L^2(\Omega)} + \| \tilde \lambda - \bar \Lambda \|_{L^2(\Omega)}.
\]
We now use the definition of $\tilde \lambda$, given by \eqref{eq:auxiliary_variables}, and the fact that $\Pi_{[-1,1]}$, defined in \eqref{def:Pi}, is a \EO{Lipschitz-1 continuous} map, to arrive at
\begin{equation}
\begin{aligned}
\label{eq:subgradient_error}
 \| \bar t - \bar \Lambda \|^2_{L^2(\Omega)} 
 & \lesssim  \| \Pi_{[-1,1]} (-\tfrac{1}{\nu} \tr \bar p) - \Pi_{[-1,1]} (-\tfrac{1}{\nu} \tr \bar P) \|^2_{L^2(\Omega)} +  \E_{\Lambda}^2(\bar{\Lambda},\bar{P}; \T_{\Omega})
 \\
 &\lesssim  \| \tr( \bar p - \bar P) \|^2_{L^2(\Omega)} +  \E_{\Lambda}^2(\bar{\Lambda},\bar{P}; \T_{\Omega}) \lesssim \E_{\textrm{ocp}}(\bar{V},\bar{P},\bar{Z},\bar{\Lambda}; \T_{\Y}). 
 \end{aligned}
\end{equation}

Step 5. The desired estimate \eqref{eq:reliability} follows from \eqref{eq:control_error} %%\eqref{eq:state_error},\eqref{eq:adjoint_error}, 
and \eqref{eq:subgradient_error}. This concludes the proof.
\end{proof}

\subsubsection{Efficiency}

We now derive the local efficiency of the \emph{ideal} error estimator $\E_{\textrm{ocp}}$. In order to present the next result we define the constant $\mathfrak{C}(d_s,\sigma,\nu)$, that depends only on $d_s$, the regularization parameter $\sigma$, and the sparsity parameter $\nu$, by
\begin{equation}
\label{eq:C}
 \mathfrak{C}(d_s,\sigma,\nu) = \max \{  \EO{2} , \EO{d_s^{\frac{1}{2}}} + 1, d_s^{-\frac{1}{2}}( \nu^{-1} + 2 \sigma^{-1} + \EO{d_s^{\frac{1}{2}}}), 1 \}.
\end{equation}

\begin{theorem}[local lower bound]
\label{thm:ideal_2}
Let $(\bar{v}, \bar{p}, \orsf,\bar t) \in \HL(y^{\alpha},\C_{\Y}) \times \HL(y^{\alpha},\C_{\Y}) \times \Zad \times L^{\infty}(\Omega)$ be the optimal variables associated to the truncated optimal control problem of Section \ref{subsec:truncated} and $(\bar{V},\bar{P},\bar{Z},\bar \Lambda) \in \V(\T_{\Y}) \times \V(\T_{\Y}) \times \mathbb{Z}_{ad}(\T_{\Omega}) \times \mathbb{P}_0(\T_{\Omega})$ its numerical approximation as is described in Section \ref{subsec:fully}. If $z' \in \N(\T_{\Omega})$, then
\begin{multline}
\label{eq:efficiency}
 \E_{V}(\bar{V},\bar{Z}; \C_{z'} ) + \E_{P}(\bar{P},\bar{V}; \C_{z'} ) +  \E_{Z}(\bar{Z},\bar{P}; S_{z'}) + \E_{\Lambda}(\bar{\Lambda},\bar{P}; S_{z'}) \leq \mathfrak{C}(d_s,\sigma,\nu) 
 \\ 
 \cdot \left( \| \nabla e_{V} \|_{L^2(y^{\alpha},\C_{z'} )}  +  \| \nabla e_P \|_{L^2(y^{\alpha},\C_{z'})} + \| e_Z \|_{L^2(S_{z'})} + \| e_{\Lambda} \|_{L^2(S_{z'})}  \right),
\end{multline}
where $\mathfrak{C}(d_s,\sigma,\nu)$ depends only on $d_s$, $\sigma$, and $\nu$ and is defined in \eqref{eq:C}.
\label{th:global_efficiency}
\end{theorem}
\begin{proof} We proceed in five steps.

Step 1. We study the efficiency properties of the local indicator $\E_{V}$ defined in \eqref{eq:defofEV}. Let $z' \in \N(\T_{\Omega})$. Since $\zeta_{z'} \in \W(\C_{z'})$ solves \eqref{eq:ideal_local_problemV} and $\langle \orsf, \tr\zeta_{z'} \rangle = a_{\Y}( \bar v, \zeta_{z'}) = a_{z'}( \bar v, \zeta_{z'})$, which follows from setting $\phi = \zeta_{z'}$ in \eqref{eq:alpha_harm_truncated}, we obtain that
% We invoke the fact that $\zeta_{z'}$ solves the local problem \eqref{eq:ideal_local_problemV} and conclude that
\begin{align}
\nonumber
\EO{d_s^{-1}}\E_{V}^2(\bar V, \bar Z; \C_{z'})  = a_{z'}( \zeta_{z'}, \zeta_{z'}) & = \langle \orsf, \tr\zeta_{z'} \rangle + \langle \bar Z - \orsf, \tr\zeta_{z'} \rangle - a_{z'}( \bar V, \zeta_{z'})
\\
& = a_{z'}( \bar v - \bar V, \zeta_{z'}) + \langle \bar Z - \orsf, \tr\zeta_{z'} \rangle.
\label{eq:efficiency_V_before}
\end{align}
Use now the local version of the trace estimate \eqref{Trace_estimate_local} with $C_{\tr} \leq d_s^{-\frac{1}{2}}$ and the definition of the local bilinear form $a_{z'}$, given by \eqref{eq:a_local}, to conclude that
\begin{align*}
\EO{d_s^{-1}}\E_{V}^2(\bar V, \bar Z; \C_{z'}) 
% & \leq d_s^{-1}\|  \nabla e_{V} \|_{L^2(y^{\alpha},\C_{z'})}\| \nabla \zeta_{z'}\|_{L^2(y^{\alpha},\C_{z'})} + \| \orsf -\bar{Z} \|_{L^2(S_{z'})} \| \tr \zeta_{z'}\|_{L^2(S_{z'})}
%  \\
%  & 
 \leq \left( d_s^{-1}\|  \nabla e_{V} \|_{L^2(y^{\alpha},\C_{z'})} + d_s^{-\frac{1}{2}}\| e_{Z} \|_{L^2(S_{z'})} \right) \|  \nabla \zeta_{z'}\|_{L^2(y^{\alpha},\C_{z'})},
\end{align*}
which, on the basis of \eqref{eq:defofEV}, immediately yields the local efficiency of $\E_{V}$:
\begin{equation}
\label{eq:efficiency_V}
\E_{V} (\bar V, \bar Z; \C_{z'}) \leq \|  \nabla e_{V} \|_{L^2(y^{\alpha},\C_{z'})} + \EO{d_s^{\frac{1}{2}}} \| e_{Z} \|_{L^2(S_{z'})}.
\end{equation}

Step 2. Similar arguments to the ones that lead to \eqref{eq:efficiency_V_before} reveal that
\[
 \EO{d_s^{-1}}\E_{P}^2(\bar P, \bar V; \C_{z'}) =  a_{z'}( \chi_{z'},e_{P}) - \langle \tr e_{V}, \tr \chi_{z'}  \rangle,
\]
where $\chi_{z'} \in \W(\C_{z'})$ solves \eqref{eq:ideal_local_problemP}. Thus,
\begin{equation}
\label{eq:efficiency_P}
\E_{P} (\bar P, \bar V; \C_{z'}) \leq  \| \nabla e_{P} \|_{L^2( y^{\alpha}, \C_{z'})} + \|  \nabla e_{V} \|_{L^2(y^{\alpha},\C_{z'})}.
\end{equation}

Step 3. We now derive local efficiency properties for the indicator $\E_{\Lambda}$, given by \eqref{eq:defofELambdaglobal}. To accomplish this task, we invoke the projection formula \eqref{eq:projection_t} for $\bar t$ and the definition of $\tilde \lambda$, given by \eqref{eq:auxiliary_variables}, to obtain, for $z' \in \N(\T_{\Omega})$, that
\begin{equation*}
  \E_{\Lambda}(\bar{\Lambda},\bar{P}; S_{z'}) \leq \| \bar \Lambda - \bar t   \|_{L^2(S_{z'})} +  \| \Pi_{[-1,1]} ( -\tfrac{1}{\nu} \tr \bar p) -  \Pi_{[-1,1]} ( - \tfrac{1}{\nu} \tr \bar P)   \|_{L^2(S_{z'})}.
\end{equation*}
This, in view of the \EO{Lipschitz-1 continuity} of the operator $\Pi_{[-1,1]}$ and the local trace estimate \eqref{Trace_estimate_local}, implies that
\begin{equation}
  \E_{\Lambda}(\bar{\Lambda},\bar{P}; S_{z'}) 
%   \leq 
%   \| e_{\Lambda}  \|_{L^2(S_{z'})} +  \tfrac{1}{\nu}\| \tr  e_{P} \|_{L^2(K)} 
  \leq \| e_{\Lambda}  \|_{L^2(S_{z'})} +   \nu^{-1} d_s^{-\frac{1}{2}}\| \nabla e_P \|_{L^2(y^{\alpha},\C_{z'})}.
  \label{eq:efficiency_Lambda}
\end{equation}

Step 4. In this step we study the efficiency properties of $\E_{Z}$, which is defined by \eqref{eq:defofEZglobal}. Let $z' \in \N(\T_{\Omega})$. Invoke the projection formula for $\bar \rsf$ and the definition of $\tilde \rsf$, given by \eqref{eq:projection_r} and \eqref{eq:auxiliary_variables}, respectively, to arrive at
\begin{multline*}
 \E_{Z}(\bar{Z},\bar{P}; S_{z'}) \leq  \| \bar Z - \bar \rsf \|_{L^2(S_{z'})} 
 \\
 + \left \| 
  \Pi_{[\asf,\bsf]} \left[ -\tfrac{1}{\sigma} \left( \tr \bar p +\nu \bar t  \right) \right] - 
  \Pi_{[\asf,\bsf]} \left[ -\tfrac{1}{\sigma} \left( \tr \bar P +\nu \tilde \lambda \right) \right] 
 \right\|_{L^2(S_{z'})},
\end{multline*}
which, in view of the \EO{Lipschitz-1 continuity} of $\Pi_{[\asf,\bsf]}$, implies that
\begin{equation*}
 \E_{Z}(\bar{Z},\bar{P}; S_{z'}) \leq \| \bar Z - \bar \rsf \|_{L^2(S_{z'})} + \frac{1}{\sigma} \| \tr (\bar p - \bar P) \|_{L^2(S_{z'})} + \frac{\nu}{\sigma} \| \bar t - \tilde \lambda \|_{L^2(S_{z'})}.
\end{equation*}
We now invoke \eqref{Trace_estimate_local} and the fact that $\tilde \lambda = \Pi_{[-1,1]}(-\frac{1}{\nu} \tr \bar P)$ to arrive at
\begin{equation}
\E_{Z}(\bar{Z},\bar{P}; S_{z'}) \leq \| e_{Z}  \|_{L^2(S_{z'})}  + 2 \sigma^{-1} d_s^{-\frac{1}{2}} \| \nabla e_P \|_{L^2(y^{\alpha},\C_{z'})}.
\label{eq:efficiency_Z}
\end{equation}

Step 5. We conclude the proof of the local efficiency property \eqref{eq:efficiency} by collecting the estimates \eqref{eq:efficiency_V}, \eqref{eq:efficiency_P}, \eqref{eq:efficiency_Lambda}, and \eqref{eq:efficiency_Z}.
\end{proof}

\begin{remark}[Local efficiency of each contribution]\rm
The arguments elaborated in the proof of Theorem \ref{th:global_efficiency} reveal that each contribution $\E_{V}$, $\E_{P}$, $\E_{\Lambda}$, and $\E_{Z}$ is locally efficient; see estimates \eqref{eq:efficiency_V}--\eqref{eq:efficiency_Z}. In addition, the constants involved in such efficiency results are \emph{fully computable} and depend only on $d_s$, $\nu$, and $\sigma$.
 \end{remark}

%%%%%%%%%%%%%%%%%%%%%%%%%%%%%%%%%%%%%%%%%%%%%%%%%%%%%%%%%%%%%%%%%%%%%%%%%%%%%%%%%%%%%%%%%%%%%%%%%%%%%%%%%%%%%%%%%%%%%%%%%%%%%%%%%%%%%%%%%%%%%%%%%%%%%%%%%%%%%%%%%%%%%%%%%%%
\section{A computable a posteriori error estimator}
\label{subsec:computable_a_posteriori}
%%%%%%%%%%%%%%%%%%%%%%%%%%%%%%%%%%%%%%%%%%%%%%%%%%%%%%%%%%%%%%%%%%%%%%%%%%%%%%%%%%%%%%%%%%%%%%%%%%%%%%%%%%%%%%%%%%%%%%%%%%%%%%%%%%%%%%%%%%%%%%%%%%%%%%%%%%%%%%%%%%%%%%%%%%%

The results of Theorems \ref{thm:ideal_1} and \ref{thm:ideal_2} show that $\E_{ocp}$ is globally reliable and locally efficient \EO{with no oscillation terms.} Consequently, $\E_{ocp}$ is equivalent to the error $\VERT e \VERT$. However, it has an insurmountable drawback: it requires, for each node $z'$, the computation of the contributions $\E_{V}$ and $\E_{P}$ that in turn require the resolution of the local problems \eqref{eq:ideal_local_problemV} and \eqref{eq:ideal_local_problemP}, respectively. Since these problems are posed on the infinite--dimensional space $\W(\C_{z'})$, the error estimator \eqref{eq:defofEocp} is not computable. In spite of this fact, it provides the intuition to define an anisotropic and computable posteriori error estimator.

Let us begin our discussion with the following definition.

\begin{definition}[discrete local spaces]
\label{def:discrete_spaces}
For $z' \in \N(\T_\Omega)$, we define
\begin{align*}
  \mathcal{W}(\C_{z'}) = & 
    \left\{
      W \in C^0( \overline{\C_{z'}} ): W|_T \in \mathcal{P}_2(K) \otimes \mathbb{P}_2(I) \ \forall T = K \times I \in \C_{z'}, \right. \\
    &\left.  W|_{\partial \C_{z'} \setminus \Omega \times \{ 0\} } = 0
    \right\}.
\end{align*}
If $K$ is a quadrilateral, $\mathcal{P}_2(K)$ stands for $\mathbb{Q}_2(K)$. When $K$ is a simplex, $\mathcal{P}_2(K)$ corresponds to $\mathbb{P}_2(K) \oplus \mathbb{B}(K)$, where $\mathbb{B}(K)$ stands for the space spanned by a local 
%cubic 
bubble function \EO{\ie $\mathbb{B}(K) = \mathrm{span}(b)$, where 
$
b = (n+1)^{n+1} \lambda_1 \dots \lambda_{n+1}
$
and $\{ \lambda_1, \dots, \lambda_{n+1}\}$ are the barycentric coordinates of $K$.}
\end{definition}

With the discrete space $\mathcal{W}(\C_{z'})$ at hand, we define the discrete functions $\eta_{z'}$ and $\theta_{z'}$ as follows:
\begin{equation}
\label{eq:ideal_local_problemV_computable}
\eta_{z'} \in \mathcal{W}(\C_{z'}): 
\quad 
a_{z'}( \eta_{z'}, W)  = \langle \bar{Z}, \tr W  \rangle  -  a_{z'}(\bar{V}, W)
\end{equation}
for all $W \in \Wcal(\C_{z'})$ and
\begin{equation}
\label{eq:ideal_local_problemP_computable}
\theta_{z'} \in \mathcal{W}(\C_{z'}): 
\quad 
a_{z'}( W,\theta_{z'} )= \langle \tr \bar{V} - \usf_d, \tr W \rangle  - a_{z'}( W,\bar{P} )
\end{equation}
for all $W \in \Wcal(\C_{z'})$. Notice that problems \eqref{eq:ideal_local_problemV_computable} and \eqref{eq:ideal_local_problemP_computable} correspond to the Galerkin discretization of problems \eqref{eq:ideal_local_problemV} and \eqref{eq:ideal_local_problemP}, respectively.

We define the computable counterpart of $\E_{\mathrm{ocp}}$, defined in \eqref{eq:defofEocp}, as follows:
\begin{multline}
 \mathcal{E}^2_{\textrm{ocp}}(\bar{V},\bar{P},\bar{Z},\bar{\Lambda}; \T_{\Y}) = \mathcal{E}^2_{V}(\bar{V},\bar{Z}; \N(\T_{\Omega})) + \mathcal{E}^2_{P}(\bar{P},\bar{V}; \N(\T_{\Omega})) 
 \\
 + \mathcal{E}^2_{Z}(\bar{Z},\bar{P}; \T_{\Omega}) + \mathcal{E}^2_{\Lambda}(\bar{\Lambda},\bar{P}; \T_{\Omega}).
  \label{eq:defofEocp_computable}
\end{multline}
We recall that $\bar{V}$, $\bar{P}$, $\bar{Z}$, and $\bar{\Lambda}$ denote the optimal variables associated to the fully discrete optimal control problem of Section \ref{subsec:fully}. 

We now describe each contribution in \eqref{eq:defofEocp_computable}. First, we define the local error indicators and error estimator associated to the state equation by
\begin{equation}
\label{eq:defofEV_computable_local} 
\mathcal{E}_{V}^2(\bar{V},\bar{Z}; \C_{z'}) = \| \nabla \eta_{z'} \|^2_{L^2(y^{\alpha},\C_{z'})}
\end{equation}
and 
\begin{equation}
\label{eq:defofEV_computable_global} 
\mathcal{E}_{V}^2(\bar{V},\bar{Z}; \N(\T_{\Omega})) =  \sum_{z' \in \N(\T_{\Omega})} \mathcal{E}_{V}^2(\bar{V},\bar{Z}; \C_{z'}),
\end{equation}
respectively. 
% and the global error estimator 
% $
% E_{V}(\bar{V},\bar{Z}; \N(\T_{\Omega})) := \left( \sum_{z' \in \N(\T_{\Omega})} E_{V}^2(\bar{V},\bar{Z}; \C_{z'}) \right)^{\frac{1}{2}}.
% $ 
Second, we define the local error indicators and error estimator associated to the discretization of the adjoint equation as
% The second contribution in \eqref{eq:defofEocp_computable}, that is associated to the discretization of the adjoint equation, is defined similarly. We define the local error indicators and error estimator as
\begin{equation}
\label{eq:defofEP_computable_local} 
\mathcal{E}_{P}^2(\bar{P},\bar{V}; \C_{z'}) = \| \nabla \theta_{z'} \|^2_{L^2(y^{\alpha},\C_{z'})}
\end{equation}
and \begin{equation}
\label{eq:defofEP_computable_global} 
\mathcal{E}_{P}^2(\bar{P},\bar{V}; \N( \T_{\Omega} ) ) = \sum_{z' \in \N( \T_{\Omega} )} \mathcal{E}_{P}^2(\bar{P},\bar{V}; \C_{z'}),
\end{equation}
respectively. The third and fourth contributions, $\mathcal{E}_{Z}$ and $\mathcal{E}_{\Lambda}$, are defined exactly as in \eqref{eq:defofEZglobal} and \eqref{eq:defofELambdaglobal}, respectively.

\subsection{Efficiency}

The next result shows the local efficiency of the computable a posteriori error estimator $\mathcal{E}_{\textrm{ocp}}$. 

\begin{theorem}[local lower bound]
Let $(\bar{v}, \bar{p}, \orsf, \bar t) \in \HL(y^{\alpha},\C_{\Y}) \times \HL(y^{\alpha},\C_{\Y}) \times \Zad \times L^{\infty}(\Omega)$ be the optimal variables associated to the truncated optimal control problem of Section \ref{subsec:truncated} and $(\bar{V},\bar{P},\bar{Z}, \bar{\Lambda}) \in \V(\T_{\Y}) \times \V(\T_{\Y}) \times \mathbb{Z}_{ad}(\T_{\Omega}) \times \mathbb{P}_0(\T_{\Omega})$ its numerical approximation as described in Section \ref{subsec:fully}. If $z' \in \N(\T_{\Omega})$, then
\begin{multline}
\label{eq:efficiency_computable}
 \mathcal{E}_{V}(\bar{V},\bar{Z}; \C_{z'} ) + \mathcal{E}_{P}(\bar{P},\bar{V}; \C_{z'} ) +  \mathcal{E}_{Z}(\bar{Z},\bar{P}; S_{z'} ) + \mathcal{E}_{\Lambda}(\bar{\Lambda},\bar{P}; S_{z'} )  
 \leq \mathfrak{C}(d_s,\sigma,\nu) 
 \\
 \cdot \left( \| \nabla e_{V} \|_{L^2(y^{\alpha},\C_{z'})}  +  \| \nabla e_{P} \|_{L^2(y^{\alpha},\C_{z'})} + \| e_{Z} \|_{L^2(S_{z'})} + \| e_{\Lambda} \|_{L^2(S_{z'})} \right),
\end{multline}
where $\mathfrak{C}(d_s,\sigma,\nu)$ depends only on $d_s$, $\sigma$ and $\nu$ and is defined as in \eqref{eq:C}.
\label{th:global_efficiency_computable}
\end{theorem}
\begin{proof}
\EO{The local efficiency properties of the estimators $\mathcal{E}_{\Lambda}$ and $\mathcal{E}_{Z}$ follow from \eqref{eq:efficiency_Lambda} and \eqref{eq:efficiency_Z}, respectively; we recall that $\E_{\Lambda} = \mathcal{E}_{\Lambda}$ and $\E_{Z} = \mathcal{E}_Z$. To derive the efficiency of $\mathcal{E}_{V}$, we notice that, for $z' \in \N(\T_{\Omega})$,
\[
d_s^{-1}E_{V}^2(\bar V, \bar Z; \C_{z'}) = a_{z'}( \eta_{z'}, \eta_{z'}) = 
% \langle \orsf, \eta_{z'}\rangle 
a_{z'}( \bar v - \bar V, \eta_{z'}) + \langle \bar Z - \orsf, \eta_{z'}\rangle .
\]
Invoke \eqref{Trace_estimate_local} and \eqref{eq:defofEV_computable_local} to conclude. The local efficiency of $\mathcal{E}_{P}$ follows similar arguments. For brevity, we skip details.}
% 
% The proof of the estimate \eqref{eq:efficiency_computable} repeats the arguments developed in the proof of Theorem \ref{th:global_efficiency}. We analyze the local efficiency of the indicator $E_V$ defined in \eqref{eq:defofEV_computable}. To do this, we let $z' \in \N(\T_{\Omega})$. Employing the fact that $\eta_{z'}$ solves problem \eqref{eq:ideal_local_problemV_computable} and recalling that $\orsf$ denotes the optimal control, we arrive at 
% \[
% E_{V}^2(\bar V, \bar Z; \C_{z'}) = a_{z'}( \eta_{z'}, \eta_{z'}) = \langle \orsf, \eta_{z'}\rangle + \langle \bar Z - \orsf, \eta_{z'}\rangle - a_{z'}( \bar V, \eta_{z'}).
% \]
% Invoking the trace estimate \eqref{Trace_estimate_local} with constant $C_{\tr} \leq d_s^{-\frac{1}{2}}$, the fact that $\bar v$ solves problem \eqref{eq:alpha_harm_truncated}
% and the Cauchy-Schwarz inequality, we obtain
% % Define $e_V = \bar{v} - \bar{V}$, 
% \[
% E_{V}^2(\bar V, \bar Z; \C_{z'}) \leq \left( d_s^{-1}\|  \nabla (\bar v - \bar V)\|_{L^2(y^{\alpha},\C_{z'})} + d_s^{-\frac{1}{2}}\| \orsf -\bar{Z} \|_{L^2(S_{z'})} \right) \|  \nabla \eta_{z'}\|_{L^2(y^{\alpha},\C_{z'})},
% \]
% which, in light of \eqref{eq:defofEV_computable}, immediately yields the desired result
% \begin{equation*}
% \label{eq:EV_locally_efficient}
% E_{V}(\bar V, \bar Z; \C_{z'}) \leq d_s^{-1}\|  \nabla (\bar v - \bar V)\|_{L^2(y^{\alpha},\C_{z'})} + d_s^{-\frac{1}{2}}\| \orsf -\bar{Z} \|_{L^2(S_{z'})}.
% \end{equation*}
% 
% The efficiency analysis for the terms $E_{P}$ and $E_{Z}$ follow similar arguments. For brevity we skip the proof.
\end{proof}

\begin{remark}[strong efficiency]
\rm
The lower bound \eqref{eq:efficiency_computable} does not involve any oscillation term and, in addition, is fully computable since $\mathfrak{C}(d_s,\sigma,\nu)$ is known. These properties imply a strong concept of efficiency: the relative size of the local error indicator dictates mesh refinement, which is independent of fine structure of the data. 
\end{remark}

\subsection{Reliability}

We now analyze the reliability properties of $\mathcal{E}_{ocp}$. To accomplish this task, we introduce, for $z' \in \N(\T_\Omega)$,
the local \emph{oscillation} of $f \in L^2(\Omega)$:
\begin{equation}
\label{eq:defoflocosc}
  \osc(f; S_{z'}) = h_{z'}^{s} \| f - f_{z'} \|_{L^2(S_{z'})}, \quad f_{z'}|_K = \fint_{K} f,
\end{equation}
where $h_{z'} = \min\{h_{K}: K \ni z' \}$. We also define the global oscillation of $f$ as
\begin{equation}
\label{eq:defofglobosc}
  \osc^2(f;\T_\Omega) = \sum_{z' \in \N(\T_\Omega)} \osc^2( f; S_{z'} ).
\end{equation}
With these definitions at hand, we define, 
% for $D = (\usf_d, \tr \bar V)$, we define the local oscillation of $D$ as
% \begin{equation}
% \label{eq:defofgloboscD}
%   \osc(D ;S_{z'}) = \osc(\usf_d ;S_{z'})
%   + \osc(\tr \bar V ;S_{z'}).
% %   + \osc(\bar Z ;S_{z'}).
% \end{equation}
% The global oscillation term $\osc(D ; \T_{\Omega})$ is defined accordingly. We now define, 
for $z' \in \N(\T_{\Omega})$, the \emph{total error indicator}
\begin{equation}
\label{total_error}
\mathfrak{E}^2(\bar V, \bar P, \bar Z, \bar \Lambda ; \C_{z'}) =  \mathcal{E}^2_{\textrm{ocp}}(\bar V, \bar P, \bar Z, \bar \Lambda; \C_{z'}) + \osc^2(\tr \bar{V} - \usf_d  ; S_{z'}),
\end{equation}
which will be essential in the module \textsf{\textup{MARK}} of the AFEM of Section \ref{subsec:design}. 

Let $\mathscr{K}_{\T_{\Omega}} = \{ S_{z'} : z' \in \N(\T_\Omega)\}$. We set $\mathscr{K}_{\Y} = \mathscr{K}_{\T_{\Omega}} \times (0,\Y)$ and, for any $\mathscr{M} \subset \mathscr{K}_{\T_{\Omega}}$, $\mathscr{M}_{\Y} = \mathscr{M} \times (0,\Y)$. Define 
\begin{equation}\label{total_est}
\mathfrak{E}^2( \bar V, \bar P, \bar Z, \bar \Lambda; \mathscr{M}_{\Y} ) = 
 \sum_{S_{z'} \in \mathscr{M} } \mathfrak{E}^2(\bar V, \bar P, \bar Z, \bar \Lambda; \C_{z'} ),
\end{equation}
where, we recall that $\C_{z'} = S_{z'} \times (0,\Y)$. 

\begin{remark}[Marking]\rm
We follow \cite[Remark 4.4]{MR2875241} and comment that, in the AFEM of Section \ref{sec:numerics} we will utilize the total error indicator \eqref{total_error}, namely the sum of the computable error estimator and the oscillation term, for marking. The oscillation cannot be removed for marking without further assumptions. We refer the reader to \cite{MR2875241} for a complete discussion on this matter. 
\end{remark}

\EO{We now apply the results of \cite[Theorem 5.37]{CNOS2} to obtain a posteriori error estimates that will be of importance in what follows. Let us assume that \eqref{condition} and \cite[Conjecture 5.28]{CNOS2} hold, then, since $\bar V$ corresponds to the finite element approximation of $\mathpzc{v}$, defined in \eqref{eq:v_star}, within the space $\V(\T_{\Y})$, we have 
\begin{equation}
\label{eq:aposteriori_mathpzcv}
 \| \nabla (\mathpzc{v} - \bar V) \|_{L^2(y^{\alpha},\C_{\Y})}  \lesssim \mathcal{E}_{V}(\bar{V},\bar{Z}; \N(\T_{\Omega})) +  \osc( \bar{Z};\T_{\Omega}) = \mathcal{E}_{V}(\bar{V},\bar{Z}; \N(\T_{\Omega})).
\end{equation}
Notice that the oscillation of $\bar Z$ vanishes since it is a piecewise constant function on $\T_{\Omega}$.
Similarly, we have that}
\begin{equation}
\label{eq:aposteriori_mathpzcp}
 \| \nabla (\mathpzc{p} - \bar P) \|_{L^2(y^{\alpha},\C_{\Y})}  \lesssim \mathcal{E}_{P}(\bar{P},\bar{V}; \N(\T_{\Omega})) +  \osc(\tr \bar{V} - \usf_d;\T_{\Omega}).
\end{equation}

\EO{The following comments are in order. First, the condition \eqref{condition} is fully implementable and it is indeed implemented in the the module \textsf{\textup{REFINEMENT}} of our AFEM of Section \ref{subsec:design}. Second, the reliability estimate of \cite[Theorem 5.37]{CNOS2} relies on the existence of a linear operator $\mathcal{M}_{z'}:\mathbb{W}(\C_{z'}) \rightarrow \mathcal{W}(\C_{z'})$ such that satisfies the conditions stated in \cite[Conjecture 5.28]{CNOS2}. The construction of this operator is an open problem. Numerical experiments provide computational evidence that the upper bound \eqref{eq:aposteriori_mathpzcp} is valid without the term $\osc(\tr \bar{V} - \usf_d;\T_{\Omega})$ and thus indirect evidence of the existence of $\mathcal{M}_{z'}$ satisfying the properties stated in \cite[Conjecture 5.28]{CNOS2}.}

% \begin{remark}[Conjecture]\rm
% The proof of Theorem \ref{th:reliability_computable} is valid under the assumption that there exists an operator $\mathcal{M}_{z'}$ that verifies the conditions stipulated in \cite[Conjecture 5.28]{CNOS2}. The construction of $\mathcal{M}_{z'}$ is an open problem. The numerical experiments performed for the state equation and the ones that we perform in Section \ref{sec:numerics} provide computational evidence that the bound \eqref{eq:reliability_computable} holds with no oscillation terms and thus indirect evidence of the existence of the aforementioned operator $\mathcal{M}_{z'}$ with the requisite properties \cite[(5.29)--(5.31)]{CNOS2}.
% \end{remark}

We now derive the global reliability of the computable error estimator $\mathcal{E}_{ocp}$. 
%%The proof of such a result follows the same arguments to the ones elaborated in the proof of Theorem \ref{thm:ideal_1}. The main difference is that the computable error indicators \eqref{eq:defofEV_computable_local}--\eqref{eq:defofEV_computable_global} and \eqref{eq:defofEP_computable_local}--\eqref{eq:defofEP_computable_global} are used instead of \eqref{eq:defofEV} and \eqref{eq:defofEP}, respectively. In what follows we provide a brief argument that reveals the role of the oscillation term $\osc(\tr \bar{V} - \usf_d  ; \T_{\Omega})$.

\begin{theorem}[global upper bound]
Let $(\bar{v}, \bar{p}, \orsf, \bar t) \in \HL(y^{\alpha},\C_{\Y}) \times \HL(y^{\alpha},\C_{\Y}) \times \Zad \times L^{\infty}(\Omega)$ be the optimal variables associated to the truncated optimal control problem of Section \ref{subsec:truncated} and $(\bar{V},\bar{P},\bar{Z}, \bar{\Lambda}) \in \V(\T_{\Y}) \times \V(\T_{\Y}) \times \mathbb{Z}_{ad}(\T_{\Omega}) \times \mathbb{P}_0(\T_{\Omega})$ its numerical approximation as described in Section \ref{subsec:fully}. If \eqref{condition} and \cite[Conjecture 5.28]{CNOS2} hold, then
\begin{equation} 
 \VERT e \VERT 
 \lesssim \mathfrak{E}(\bar{V},\bar{P},\bar{Z}, \bar \Lambda; \mathscr{K}_{\Y}),
 \label{eq:reliability_computable}
\end{equation}
where the hidden constant is independent of the continuous and discrete optimal variables, the size of the elements in the meshes $\T_{\Omega}$ and $\T_{\Y}$, $\# \T_{\Omega}$, and $\# \T_{\Y}$.
\label{th:reliability_computable}
\end{theorem}
\begin{proof}
\EO{We first estimate $\mathrm{II} = (\tr(\mathpzc{p} - \bar P), \tilde \rsf - \bar \rsf )_{L^2(\Omega)}$. Invoke \eqref{eq:aposteriori_mathpzcp} to arrive at
\begin{align*}
|\mathrm{II}| & \lesssim 
\| \nabla( \mathpzc{p}- \bar{P} ) \|_{L^2(y^{\alpha},\C_{\Y})} \| \tilde{\rsf} - \orsf \|_{L^2(\Omega)} 
\\
& \lesssim \left[ \mathcal{E}_{P}(\bar{P},\bar{V}; \N(\T_{\Omega})) +  \osc(\tr \bar{V} - \usf_d;\T_{\Omega}) \right]  \| \tilde{\rsf} - \orsf \|_{L^2(\Omega)}
\\
& \leq \frac{\sigma}{4}  \| \tilde{\rsf} - \orsf \|^2_{L^2(\Omega)} + C \left[ \mathcal{E}_{P}^2(\bar{P},\bar{V}; \N(\T_{\Omega})) +  \osc^2(\tr \bar{V} - \usf_d;\T_{\Omega}) \right],
\end{align*}
where $C$ denotes a positive constant. We now write $\bar p - \mathpzc{p}  = (\bar p - \tilde p) + (\tilde p - \mathpzc{p} )$ and note that \eqref{eq:estimate_for_II_1} yields $\mathrm{I}_1:=(\tr(\bar{p} - \tilde p), \tilde{\rsf} - \orsf )_{L^2(\Omega)} \leq 0$. We now control $\mathrm{I}_2:=(\tr(\tilde p - \mathpzc{p}), \tilde{\rsf} - \orsf )_{L^2(\Omega)}$. In view of \eqref{eq:aux1}, we have
$
 | \mathrm{I}_2 | \lesssim  \|  \tr (\tilde{v}-\bar{V})  \|_{L^2(\Omega)} \|\tilde{\rsf} - \orsf  \|_{L^2(\Omega)}.
$
Now, a stability estimate and \eqref{eq:defofEZglobal} yield
$
\|  \tr (\tilde{v}-\mathpzc{v})  \|_{L^2(\Omega)} \lesssim \mathcal{E}_{Z}(\bar{Z},\bar{P};\T_{\Omega}).
$
On the other hand, an application of \eqref{eq:aposteriori_mathpzcv} reveals that
\[
\|  \tr (\mathpzc{v}-\bar{V})  \|_{L^2(\Omega)} \lesssim \| \nabla( \mathpzc{v}-\bar{V} ) \|_{L^2(y^{\alpha},\C_{\Y})} \lesssim \mathcal{E}_{V}(\bar{V},\bar{Z};\N(\T_{\Omega})).
\]
We thus replace the derived estimates for $\mathrm{I_1}$, $\mathrm{I}_2$, and $\mathrm{II}$ into \eqref{eq:muleq} and the obtained one into \eqref{eq:r-Z-rtilde} to arrive at}
\begin{multline}
\EO{\| e_{Z} \|^2_{L^2(\Omega)} \lesssim \mathcal{E}^2_{V}(\bar{V},\bar{Z};\N(\T_{\Omega})) + \mathcal{E}^2_{P}(\bar{P},\bar{V};\N(\T_{\Omega}))}
\\
+ \mathcal{E}^2_{Z}(\bar{Z},\bar{P};\T_{\Omega}) +\osc^2(\tr \bar{V} - \usf_d;\T_{\Omega}).
\end{multline}

\EO{Similar arguments can be used to control $\| \nabla e_V\|_{L^2(y^{\alpha},\C_{\Y})} $ and $ \| \nabla e_P \|_{L^2(y^{\alpha},\C_{\Y})}$. The control of $\| \nabla e_{\Lambda}\|_{L^2(\Omega)} $ follows from \eqref{eq:subgradient_error}. This concludes the proof.}
\end{proof}

\section{Numerical Experiments}
\label{sec:numerics}

In this section we explore the computational performance of the computable a posteriori error estimator $\mathcal{E}_{ocp}$ with a series of numerical examples. To accomplish this task, we start, in the next section, with the design of an AFEM that is based on iterations of the loop \eqref{eq:loop}. 

\subsection{Design of the AFEM}
\label{subsec:design}
We describe the four modules in \eqref{eq:loop}:

\begin{enumerate}
 \item[$\bullet$] \textsf{\textup{SOLVE}:} Given the anisotropic mesh $\T_{\Y} \in \Tr$, we compute the unique solution $(\bar{V},\bar{P},\bar{Z},\bar \Lambda) \in \V(\T_{\Y})\times \V(\T_{\Y})\times \mathbb{Z}_{ad}(\T_{\Omega}) \times \mathbb{P}_0(\T_{\Omega})$ to the fully discrete optimal control problem of Section \ref{subsec:fully}:
 \[
  (\bar V,\bar P ,\bar Z ,\bar \Lambda) = \textsf{\textup{SOLVE}}(\T_{\Y}).
 \]
The problem is solved by using the active set strategy of \cite[Algorithm 2]{MR2556849}.

\item[$\bullet$] \textsf{\textup{ESTIMATE}:} With the discrete solution at hand, we compute, for each $z' \in \N(\T_{\Omega})$, the local error indicator
\begin{align*}
 \mathcal{E}_{\textrm{ocp}}(\bar{V},\bar{P},\bar{Z}, \bar{\Lambda}; \C_{z'}) = \mathcal{E}_{V}(\bar{V},\bar{Z}; \C_{z'}) & + \mathcal{E}_{P}(\bar{P},\bar{V}; \C_{z'}) 
 \\
 & + \mathcal{E}_{Z}(\bar{Z},\bar{P}; S_{z'}) + \mathcal{E}_{\Lambda}(\bar{\Lambda},\bar{P}; S_{z'}).
\end{align*}
The local indicators $\mathcal{E}_V$, $\mathcal{E}_P$, $\mathcal{E}_Z$, and $\mathcal{E}_{\Lambda}$ are defined by \eqref{eq:defofEV_computable_local}, \eqref{eq:defofEP_computable_local}, \eqref{eq:defofEZglobal}, and \eqref{eq:defofELambdaglobal}, respectively. Once the local indicator $\mathcal{E}_{\textrm{ocp}}$ is obtained, we compute the local oscillation term
$
 \osc(\tr \bar{V} - \usf_d  ; S_{z'})
$
and then construct the total error indicator $\mathfrak{E}$, defined in \eqref{total_error}:
\begin{equation}
\label{eq:estimate}
\left\{ \mathfrak{E} (\bar V, \bar P, \bar Z, \bar{\Lambda}; S_{z'}) \right\}_{ S_{z'} \in \mathscr{K}_{\T_\Omega}} = \textsf{\textup{ESTIMATE}}(\bar V, \bar P, \bar Z, \bar \Lambda; \T_{\Y}),
 \end{equation}
where $\mathscr{K}_{\T_{\Omega}} = \{ S_{z'} : z' \in \N(\T_\Omega)\}$. Notice that, for notational convenience, we have replaced $\C_{z'}$ by $S_{z'}$ in \eqref{eq:estimate}.

\item[$\bullet$] \textsf{\textup{MARK}:} We follow the numerical evidence presented in \cite[Section 6]{KRS} and use the \emph{maximum strategy} as a marking technique for our AFEM: Given $\theta \in [0,1]$ and the set of indicators obtained in the previous step, we select the set
\[
  \mathscr{M} = \textsf{\textup{MARK}} \left( \left\{  \mathfrak{E} (\bar V, \bar P, \bar Z, \bar \Lambda; S_{z'}) \right\}_{ S_{z'} \in \mathscr{K}_{\T_\Omega}}\right) \subset \mathscr{K}_{\T_{\Omega}},
\]
that contains the stars $S_{z'}$ in $\mathscr{K}_{\T_{\Omega}}$ that satisfies
\[
  \mathfrak{E} ( \bar V, \bar P, \bar Z, \bar \Lambda; S_{z'})  \geq \theta \mathfrak{E}_{\max} (\bar V, \bar P, \bar Z,\bar \Lambda),
\]
where $\mathfrak{E}_{\max} (\bar V, \bar P, \bar Z,\bar \Lambda):= \max\{ \mathfrak{E} ( \bar V, \bar P, \bar Z, \bar \Lambda; S_{z'}): S_{z'} \in \mathscr{K}_{\T_{\Omega}} \}$.

\item[$\bullet$] \textsf{\textup{REFINEMENT}:} We bisect all the elements $K \in \T_{\Omega}$ that are contained in $\mathscr{M}$ with the newest--vertex bisection method \cite{NSV:09} and create a new mesh $\T_\Omega'$. In the truncated optimal control problem, we choose the truncation parameter $\Y$ as $\Y = 1 + (1/3)\log(\# \T_{\Omega}')$. As a consequence of this election, the approximation and truncation errors are balanced \cite[Remark 5.5]{NOS}. Next, we generate a new mesh in the extended dimension, which we denote by $\mathcal{I}_\Y'$. The latter is constructed on the basis of \eqref{eq:graded_mesh}: the number of degrees of freedom $M$ is chosen to be a sufficiently large number in order to guarantee condition \eqref{condition}. This is achieved by first designing a partition $\mathcal{I}_\Y'$ with $M \approx (\# \T_\Omega')^{1/n}$ and checking \eqref{condition}. If this condition is not satisfied, we increase the number of mesh points until we obtain the desired result. Consequently,
\[
  \T_\Y' = \textsf{\textup{REFINE}}(\mathscr{M}), \qquad \T_{\Y}':= \T_\Omega'  \otimes \mathcal{I}_\Y'.
\]
\end{enumerate}

\subsection{Implementation}
The proposed AFEM is implemented within the MATLAB software library {\it{i}}FEM~\cite{Chen.L2008c}. The stiffness matrices associated to the finite element approximation of the state and adjoint equations are assembled exactly, while the forcing terms, in both equations, are computed with the help of a quadrature formula that is exact for polynomials of degree $4$. The optimality system associated to the fully discrete control problem of Section \ref{subsec:fully} is solved by using the active--set strategy of \cite[Algorithm 2]{MR2556849}.

To compute $\eta_{z'}$, defined in \eqref{eq:ideal_local_problemV_computable}, we follow \cite[Section 6]{CNOS2}: we loop around each node $z' \in \N(\T_{\Omega})$ and collect data about the cylindrical star $\C_{z'}$. We thus assemble the small linear system in \eqref{eq:ideal_local_problemV_computable} and solve it with the built-in \emph{direct solver} of MATLAB. The integrals that involve the weight $y^{\alpha}$ and discrete functions are computed exactly, while the ones that also involve data functions are computed element--wise by a quadrature formula that is exact for polynomials of degree 7. The computation of $\theta_{z'}$, defined in \eqref{eq:ideal_local_problemP_computable}, is similar.

In the \textsf{\textup{MARK}} step, we modify the estimator from star--wise to element--wise. To do this, we first scale the nodal--wise estimator as $\E_{z'}^2 / (\# S_{z'} )$ and then, compute
\[
\E_{K}^2 := \sum_{z'\in K}\E_{z'}^2, \quad K\in \T_{\Omega}.
\]
Consequently, we have that
$\sum _{K \in \T_{\Omega} }\E_{K}^2 = \sum _{z' \in \N(\T_\Omega) } \E_{z'}^2$. 
The cell--wise data oscillation is thus defined as
\[
\osc_{K}(f)^2 := h_{K}^{2s}\|f - \bar f_K\|^2_{L^2(K)}, 
\]
where $\bar f_K = |K|^{-1} \int_K f \diff x'$. The data oscillation is  
computed using a quadrature formula that is exact for polynomials of degree 7.

\begin{remark}[condition \eqref{condition}]
\rm
 \EO{Unless specifically mentioned, all computations are done without explicitly enforcing the mesh restriction \eqref{condition}. This provides experimental evidence to support the fact that \eqref{condition} is nothing but an artifact in our proofs. Nevertheless, computations, that we do not present here, show that optimality is still retained if one imposes \eqref{condition}.}
\end{remark}

\subsection{L--shaped domain with incompatible data}

The a priori theory of \cite{OS:Sparse} relies on the assumptions that $\Omega$ is convex and $\usf_d$ belongs to $\Ws$; the latter implies that $\usf_d$ vanishes on $\partial \Omega$ for $s \in (0,\tfrac{1}{2}]$. Let us thus consider a case where these conditions are not met:

\begin{enumerate}
 \item [$\bullet$] The regularization and sparsity parameters are $\sigma =  0.1$ and $\nu = 0.5$ when $s \leq \tfrac{1}{2}$ and $\sigma =  0.05$ and $\nu = 0.1$ when $s>\tfrac{1}{2}$.
 \item [$\bullet$] The control bounds $\asf$ and $\bsf$ are such that $\asf = -0.3$ and $\bsf = 0.3$.
 \item [$\bullet$] The domain $\Omega = (-1,1)^2 \setminus (0,1) \times (-1,0)$ is non-convex.
 \item [$\bullet$] The desired state is $\usf_d = 1$. Notice that $\usf_d$ is not compatible, in the sense that $\usf_d \notin \Ws$ for $s \in (0,\tfrac{1}{2}]$. This has as a consequence that the singular behavior \eqref{eq:u_singular_at_the_boundary}--\eqref{eq:u_singular_at_the_boundary_1/2} for the optimal adjoint variable $\bar \psf$ is expected.
 \item [$\bullet$] The parameter $\theta$ that governs the module \textsf{\textup{MARK}} is $\theta = 0.5$.
\end{enumerate}

The results of Figure \ref{fig:computational_rates_L} show that the AFEM proposed in Section \ref{subsec:design} delivers optimal experimental rates of convergence for the total error estimator $\mathfrak{E}$ and all choices of the parameter $s$ considered. 

\begin{figure}[ht]
  \begin{center}
  \includegraphics[width=0.5\textwidth]{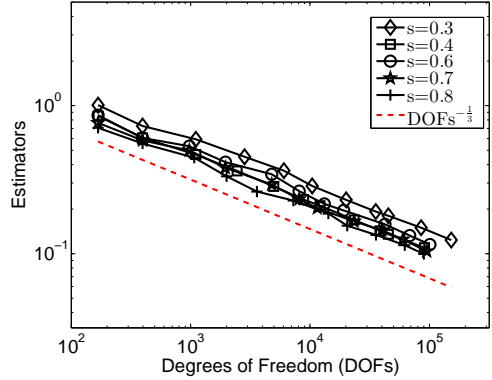}
  \end{center}
  \caption{Computational rates of convergence for our anisotropic AFEM with incompatible desired data $\usf_d$ and a non-convex domain $\Omega= (-1,1)^2 \setminus (0,1) \times (-1,0)$: $n=2$ and $s \in \{0.3,0.4,0.6,0.7,0.8 \}$. The panel shows the decrease of the total error indicator $\mathfrak{E}$ with respect to the number of degrees of freedom (DOFs). In all the presented cases the optimal rate $(\# \T_{\Y})^{-1/3}$ is achieved.}
\label{fig:computational_rates_L}
\end{figure}

In Figure \ref{fig:separated_rates} we present, for $s=0.3$ (left) and $s=0.9$ (right), the computational rates of convergence for each of the four contributions of the computable error estimator $\mathcal{E}_{ocp}$: $\mathcal{E}_{V}$, $\mathcal{E}_{P}$, $\mathcal{E}_Z$, and $\mathcal{E}_{\Lambda}$. It can be observed that, in both cases, each contribution decays with the optimal rate $\# (\T_{\Y})^{-1/3}$. 
%%%In Figure \ref{fig:meshes} we present the adaptive meshes obtained by our AFEM after 11 iterations and the corresponding finite element approximations of the optimal control variable $\ozsf$. 
\begin{figure}[ht]
  \begin{center}
  \includegraphics[width=0.45\textwidth]{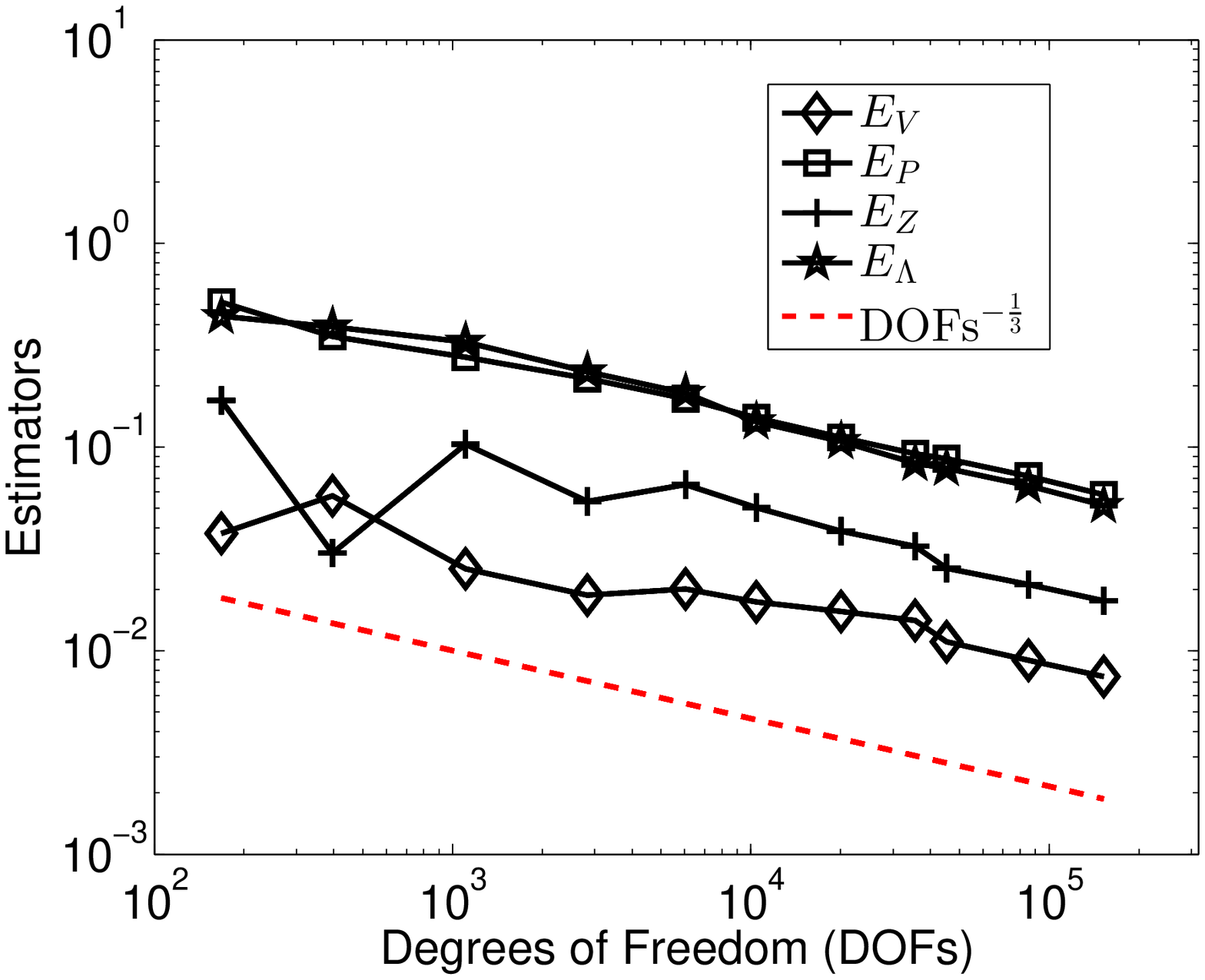}
  \hfil
  \includegraphics[width=0.45\textwidth]{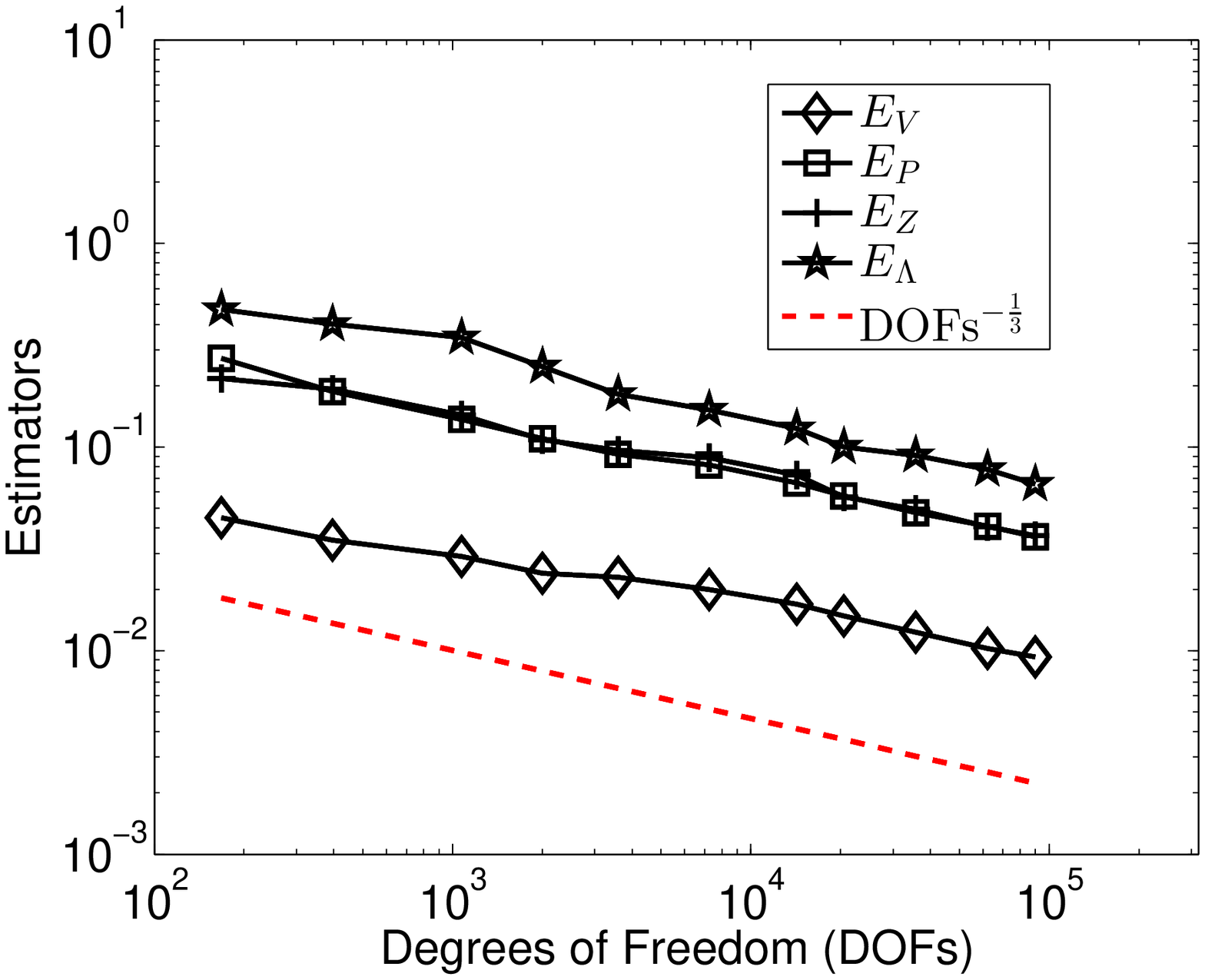}
  \end{center}
  \caption{Computational rates of convergence for the contributions $\mathcal{E}_{V}$, $\mathcal{E}_{P}$, $\mathcal{E}_Z$, and $\mathcal{E}_{\Lambda}$ of the computable and anisotropic a posteriori error estimator $\mathcal{E}_{ocp}$. We have considered $n=2$, an incompatible desired data $\usf_d$, and a non-convex domain $\Omega= (-1,1)^2 \setminus (0,1) \times (-1,0)$. The panels show the decrease of the contributions $\mathcal{E}_{V}$, $\mathcal{E}_{P}$, $\mathcal{E}_Z$, and $\mathcal{E}_{\Lambda}$ with respect to the number of degrees of freedom  for $s=0.3$ (left panel) and $s=0.8$ (right panel). In both cases we recover the optimal rate $(\# \T_{\Y})^{-1/3}$ for each contribution.}
\label{fig:separated_rates}
\end{figure}

%%%\begin{figure}[!h]
%%%\centering
%%%
%%%	\subfloat[$s = 0.2$ $(\sigma = 0.1$ and $\nu = 0.5$)]{\label{fig_s02}\includegraphics[width=1\textwidth]{fig_s02}}
%%%
%%%	\subfloat[$s = 0.3$ $(\sigma = 0.1$ and $\nu = 0.5$)]{\label{fig_s03}\includegraphics[width=1\textwidth]{fig_s03}}
%%%
%%%	\subfloat[$s=0.8$ $(\sigma = 0.05$ and $\nu = 0.1$)]{\label{fig_s08}\includegraphics[width=1\textwidth]{fig_s08}}
%%%
%%%	\subfloat[$s=0.9$ $(\sigma = 0.05$ and $\nu = 0.1$)]{\label{fig_s09}\includegraphics[width=1\textwidth]{fig_s09}}	
%%%
%%%	\caption{Adaptive meshes obtained by our anisotropic AFEM, after 11 iterations, with incompatible desired data $\usf_d$ and a non-convex domain $\Omega= (-1,1)^2 \setminus (0,1) \times (-1,0)$, and the corresponding finite element approximations of the optimal control variable $\ozsf$.}
%%%
%%%	\label{fig:meshes}
%%%
%%%\end{figure}

\EO{We now investigate the effect of varying the regularization parameter $\sigma$ by considering  $s=0.8$, $\asf = - 0.3$, $\bsf= 0.3$, $\theta = 0.5$, $\nu = 0.2$, and 
\[
\sigma \in \{ 10^0, 10^{-1}, 10^{-2}, 10^{-3}, 10^{-4} \}.
\]
In Figure \ref{fig:variando_sigma}, we observe that optimal experimental rates of convergence are obtained for the error estimator $\mathcal{E}_{ocp}$ for all the values of the parameter $\sigma$ considered.  In Figure \ref{fig_bang_bang} we present the adaptive mesh obtained by our AFEM after 9 iterations and the corresponding finite element approximations of the optimal control variable $\ozsf$ (a bang--bang controller).}

\begin{figure}[ht]
  \begin{center}
  \includegraphics[width=0.39\textwidth]{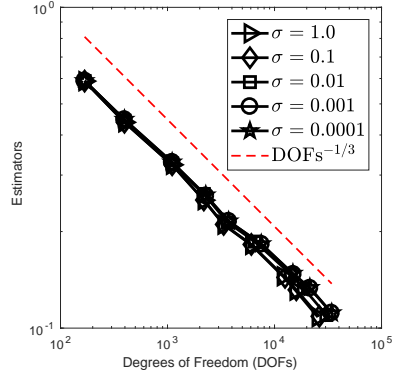}
  \end{center}
  \caption{\EO{Computational rates of convergence for our anisotropic AFEM with incompatible desired data $\usf_d$ and a non-convex domain $\Omega= (-1,1)^2 \setminus (0,1) \times (-1,0)$: $n=2$, $\nu = 0.2$, and $\sigma \in  \{ 10^0, 10^{-1}, 10^{-2}, 10^{-3}, 10^{-4} \}$. The panel shows the decrease of the total error indicator $\mathfrak{E}$ with respect to the number of degrees of freedom (DOFs). In all the presented cases the optimal rate $(\# \T_{\Y})^{-1/3}$ is achieved.}}
\label{fig:variando_sigma}
\end{figure}

\begin{figure}[!h]
 \begin{center}
\includegraphics[width=1\textwidth]{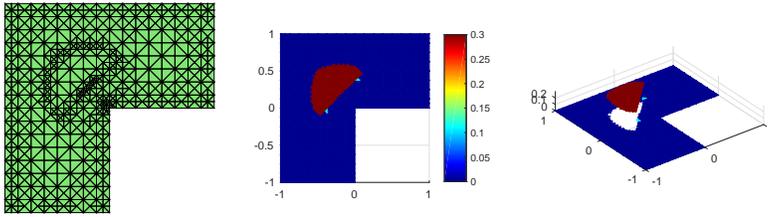}
 \end{center}
\caption{\EO{Adaptive mesh obtained by our anisotropic AFEM, after 11 iterations, with incompatible desired data $\usf_d$ and a non-convex domain $\Omega= (-1,1)^2 \setminus (0,1) \times (-1,0)$, and the corresponding finite element approximations of the optimal control variable $\ozsf$. We have considered the sparsity parameter $\nu = 0.2$ and the regularization parameter $\sigma = 10^{-4}$.}}
\label{fig_bang_bang}	
\end{figure}

\subsection{Conclusions}
Several conclusions can be drawn:

\begin{enumerate}
 \item [$\bullet$] \emph{Geometric singularities and fractional diffusion behavior}: Since $\usf_d = 1$, $\usf_d$ is an incompatible forcing term, when $s \in (0,\tfrac{1}{2}]$, for the adjoint equation $\Laps \bar \psf = \bar \usf - \usf_d$. Our results show that, when $s = 0.2$ and $s=0.3$, the devised AFEM localizes an important density of degrees of freedom near the boundary of the domain $\Omega$ \EO{and also near the reentrant corner.} This is to compensate the singular behavior \eqref{eq:u_singular_at_the_boundary}--\eqref{eq:u_singular_at_the_boundary_1/2} that is inherent to fractional diffusion with incompatible data. When $s = 0.8$ and $s = 0.9$, the incompatibility of the desired data \EO{is not longer active} and then the optimal adjoint state do not exhibits boundary layers. Our AFEM is thus focused on resolving the reentrant corner.
 \item [$\bullet$] \emph{Sparse optimal controls}: The fact that the cost functional $J$ involves the term $\| \zsf \|_{L^1(\Omega)}$ leads to sparsely supported optimal controls. Figure \ref{fig_bang_bang} is a particular instance of this feature. In addition, it can be observed that the error estimator $\mathcal{E}_{ocp}$ also focuses on refining the interface, \ie the boundary of the active set.
\end{enumerate}

\bibliographystyle{plain}
\bibliography{biblio}

\end{document}